\documentclass[reqno,11pt]{amsart}
\usepackage[utf8]{inputenc}
\usepackage{fullpage}
\usepackage{mathtools}
\usepackage{amsmath}
\usepackage{amssymb}
\usepackage{bbm}
\usepackage{tikz-network}
\usepackage{setspace}
\usepackage{cite}

\newcommand{\G}{\overrightarrow{\mathbb{G}}}
\newcommand{\E}{\overrightarrow{E}}
\newcommand{\F}[1]{\mathcal{F}_{#1}}
\newcommand{\X}[2]{X_{#1}^{(#2)}}
\newcommand{\yi}[2]{\xi_{#1}^{(#2)}}
\newcommand{\D}[2]{D_{#1}^{(#2)}}
\newcommand{\s}[2]{S_{#1}^{(#2)}}
\newcommand{\C}[2]{C_{#1,#2}}

\newcommand{\us}[2]{u_{#1}^{(#2)}}
\newcommand{\st}[2]{\hat{S}_{#1}^{(#2)}}

\newcommand{\bb}[1]{\mathbbm{#1}}

\newcommand{\Q}{\mathbb{P}}

\newcommand{\tk}{\tilde{k}}
\newcommand{\ts}{\tilde{S}}
\newcommand{\tm}{\Delta\tilde{M}}
\newcommand{\tx}{\tilde{X}}
\newcommand{\tl}{\tilde{\Lambda}}
\newcommand{\ti}{\tilde{I}}
\newcommand{\tc}[1]{\tilde{C}_{#1,n}}
\newcommand{\h}{\mathbf{H}}
\newcommand{\T}{\mathbf{T}}
\newcommand{\J}{\mathbf{J}}

\newtheoremstyle{normal}
  {1em plus .2em minus .1em} 
  {1em plus .2em minus .1em} 
  {\normalfont} 
  {} 
  {\bfseries} 
  {} 
  {.5em} 
  {} 

\newtheorem{theorem}{Theorem}

\newtheorem{lemma}{Lemma}
\newtheorem{corollary}{Corollary}
\newtheorem{proposition}{Proposition}

\theoremstyle{normal}
\newtheorem{definition}{Definition}

\newtheorem{remark}{Remark}

\usepackage{xpatch}
\makeatletter
\xpatchcmd{\@thm}{\thm@headpunct{.}}{\thm@headpunct{}}{}{}
\makeatother

\usepackage{etoolbox}
\makeatletter
\patchcmd{\@sect}{%
   \ignorespaces#8\unskip\@addpunct.}{\ignorespaces#8\unskip}{}{}
\makeatother

\title{Elephant random walks with graph based shared memory:\\ First and second order asymptotics}
\author{Deborshi Das}
\address{Theoretical Statistics and Mathematics Unit,
Indian Statistical Institute, Delhi Centre,
7 S. J. S. Sansanwal Marg,
New Delhi 110016,
INDIA}
\email{deborshidas6@gmail.com}

\begin{document}
\maketitle
\pagestyle{plain}

\begin{abstract}
We consider a generalization of the so-called elephant random walk by introducing multiple elephants moving along the integer line, $\mathbb{Z}$. When taking a new step, each elephant considers not only its own previous steps but also the past steps of other elephants. The dynamics of ``who follows whom" are governed by a directed graph, where each vertex represents an elephant, and the edges indicate that an elephant will consider the past steps of its in-neighbour elephants when deciding its next move. In other words, this model involves a collection of reinforced random walks evolving through graph-based interactions. We briefly investigate the first- and second-order asymptotic behaviour of the joint walks and establish connections with other network-based reinforced stochastic processes studied in the literature. We show that the joint walk can be expressed as a stochastic approximation scheme. In certain regimes, we employ tools from stochastic approximation theory to derive the asymptotic properties of the joint walks. Additionally, in a specific regime, we use better techniques to establish a strong invariance principle and a central limit theorem with improved rates compared to existing results in the stochastic approximation literature. These techniques can also be used to strengthen equivalent results in stochastic approximation theory. As a byproduct, we establish a strong invariance principle for the simple elephant random walk with significantly improved rates.
\end{abstract}

\section{introduction} In recent years, stochastic processes that evolve with graph-based interactive reinforcement mechanism have been gaining increasing interest in the field of probability due to their applications in various scientific fields such as biology, opinion dynamics, and statistical physics. For example, these models can capture how past outbreaks influence future disease spread in a biological system, how users' opinions in social networks (like Facebook, Reddit, and Instagram) are influenced by their friends' previous opinions and how a large system of interacting particles leads to self organised criticality (see \cite{Soc1},\cite{Soc2},\cite{Soc3},\cite{Soc4}).

\emph{Interacting urn models} are a common example of these type of processes and have been extensively studied over the past decade with various interaction mechanisms (see, for example, \cite{launay1}, \cite{launay2}, \cite{Nonlinear2}, \cite{Interactingfriedmanurn}, \cite{Kaur1}, \cite{Yogesh}, \cite{Kaur2}). Perhaps the closest to our model is the
model studied in 
\cite{Kaur1,Yogesh, Kaur2, InteractingFA-DAG}.

As in classical urn theory, the asymptotic behaviour of the proportions of ball colors in each urn has been analysed.

In this paper, we introduce a broad class of stochastic processes that evolve through a graph-based interactive framework. The primary dynamics of our joint process, denoted \( S_n \), are governed by the update rule \( S_{n+1} = S_n + X_{n+1} \), where the random variables \( \{ X_n : n \geq 1 \} \) take values in \( \{-1, 1\}^k \) for some integer \( k \geq 1 \) and meet the following conditions:

\[
\mathbb{E}[X_{n+1} \mid S_j, j \leq n] = \frac{S_n}{n} B,
\tag{*}
\]

where \( B \) is a \( k \times k \) real, non-random matrix with entries in \( [-1, 1] \), and

\[
\text{Var}(X_{n+1} \mid S_j, j \leq n) \xrightarrow{\text{a.s.}} \Theta,
\tag{**}
\]

where \( \Theta \) is a \( k \times k \) non-random positive definite matrix. Importantly, condition \((*)\) implies reinforcement via uniform sampling from past states, a reinforcement mechanism frequently applied in \emph{urn models} (see, for example, \cite{Polyaurn}, \cite{athreya1968embedding}, \cite{gouet1989martingale}, \cite{gouet1997strong}), \cite{laruelle2013randomized}, \cite{dasgupta2011strong}, and \cite{janson2024almost}). Similarly, this reinforcement approach is relevant to \emph{interacting urn models}, as discussed in \cite{launay1}, \cite{launay2}, \cite{Interactingfriedmanurn}, \cite{Kaur1}, \cite{Yogesh}, and \cite{Kaur2}. Condition \((**)\) ensures that the process remains well-behaved and non-degenerate over time.

After a thorough mathematical exposition of our model, we will explore its connections to other network-based reinforced stochastic processes, as analysed in previous research (see \cite{Synchronisation}, \cite{Functionalcltfornetwork}, \cite{Empiricalmean}, \cite{Weightedempiricalmean}, \cite{Completefirstorder}, \cite{Polarization}).

We describe our joint process $S_n$ as a collection of \emph{elephant random walks}, where each elephant recalls past steps taken by other elephants based on its unique memory capacity. This approach, as we shall see, provides a clearer and more intuitive interpretation of the general model. The concept of elephant random walks has seen active development over the past two decades, beginning with the work of Schütz and Trimper in \cite{Trimper}, which introduced the model to explore the role of memory in random walks. Its formulation as a simple martingale (see \cite{Bercu}) and connections to randomized urns (see \cite{Erwurn}) have since drawn considerable attention within the probability community. We will examine how the elephants’ past steps influence each other, in contrast to the scenario where each elephant walks independently, unaffected by others’ histories.

It is important to note that the joint process \( S_n \) can be expressed as a stochastic approximation algorithm in \( \mathbb{R}^k \) as follows:
\[
\frac{S_{n+1}}{n+1} = \frac{S_n}{n} + \frac{1}{n+1} \left[ \left( \frac{S_n}{n} \right)(B - I) + \Delta M_{n+1} \right]
\]
where \( I \) is the \( k \times k \) identity matrix, and \( \Delta M_{n+1} \) is a uniformly bounded martingale difference sequence adapted to the natural filtration of \( S_n \). The stochastic approximation scheme is a powerful tool for analysing the first- and second-order asymptotics of randomized urns (see \cite{laruelle2013randomized}) as well as interacting urns (see \cite{Interactingfriedmanurn}, \cite{Kaur1}, \cite{Yogesh}, \cite{Kaur2}). In fact, any reinforced process that employs uniform selection from the past for reinforcement can be formulated as a stochastic approximation scheme. Recent studies have also highlighted interesting applications of stochastic approximation in the context of elephant random walks (see \cite{zhangelephant} and \cite{kmelephant}).

In a few cases, we employ tools from stochastic approximation to derive favourable asymptotic properties of \( S_n \). Additionally, by using improved techniques with enhanced bounds, we establish strong invariance with better rates than those provided by standard stochastic approximation, which can also be applied to enhance the results for stochastic approximation schemes. Finally, we examine a special case to illustrate the significance of our general results.

\subsection*{Notations used in the paper:} For a set \( S \), we use the notation \( |S| \) to denote its cardinality. All vectors in this paper are row vectors. We denote \( \mathbb{R}^k \) and \( \mathbb{C}^k \) as the collections of all row vectors of length \( k \) with real and complex entries, respectively. For \( x \in \mathbb{C}^k \) (or \( \mathbb{R}^k \)), we denote its \( j \)-th element by \( x^{(j)} \). The vectors \( e_1, e_2, \dots, e_k \in \mathbb{R}^k \) are the canonical basis vectors of \( \mathbb{R}^k \). The spaces \( \mathbb{C}^{k \times k} \) and \( \mathbb{R}^{k \times k} \) represent all \( k \times k \) complex and real matrices, respectively. For $\lambda\in\mathbb{C}$, we denote the real and imaginary part of $\lambda$ by $\Re(\lambda)$ and $\Im(\lambda)$ respectively.

We use \( A' \) to denote the transpose of matrix \( A \) and \( I_k \) to denote the \( k \times k \) identity matrix (sometimes dropping the subscript \( k \) when the context is clear). The notation \( A_{p,q} \) refers to the \((p, q)\)-th entry of matrix \( A \); sometimes, we define a matrix \( A \) simply by specifying its entries \( A_{p,q} \). The vector \( \mathbf{1} \) represents a row vector of all ones, while \( \mathbf{0} \) denotes a row vector of all zeros. For simplicity, we may also use \( 0 \) to denote a row vector with all elements equal to zero.

The default probability space is \( (\Omega, \mathcal{F}, \Q) \). For $p\in[0,1]$, we use the notation Rad(\( p \)) to denote the distribution that assigns probability \( p \) to \( 1 \) and \( 1 - p \) to \( -1 \); we write \( X \sim \text{Rad}(p) \) to mean that random variable \( X \) follows the Rad(\( p \)) distribution. The uniform distribution on a non-empty finite set \( S \) is denoted by Unif(\( S \)), with \( X \sim \text{Unif}(S) \) indicating that \( X \) has a Unif(\( S \)) distribution. We use \( \mathcal{N}_k(\mu, \Sigma) \) to denote a \( k \)-dimensional multivariate normal distribution with mean vector \( \mu \) and covariance matrix \( \Sigma \), sometimes omitting the subscript \( k \) when context is clear.

For \( x \in \mathbb{C}^k \) (or \( \mathbb{R}^k \)), \( \lVert x \rVert \) denotes the standard Euclidean norm of \( x \). For \( A \in \mathbb{C}^{k \times k} \) (or \( \mathbb{R}^{k \times k} \)), \( \lVert A \rVert \) denotes the standard operator norm of \( A \), which is the maximum eigenvalue of \( (\bar{A}'A)^{\frac{1}{2}} \), where \( \bar{A} \) is the complex conjugate of \( A \). Further $Sp(A)$ denotes the set of all eigenvalues of $A$. For $x\in(0,\infty)$ and $A\in\mathbb{C}^{k\times k}$, the matrix $x^A$ is defined by the series $I+\sum\limits_{j=1}^{\infty}\frac{(\log x)^j}{j!}A^j$.

The abbreviation $a.s.$ (almost surely) is used after a statement to indicate that it holds almost surely with respect to \( \Q \). All convergences mentioned in this paper are in the standard Euclidean topology of \( \mathbb{C}^{k} \). We use \( \xrightarrow{a.s.} \) and \( \xrightarrow{d} \) to denote almost sure and weak convergence, respectively. Furthermore, for $m>0$, \( X_n \xrightarrow{L^m} X \) means \( \mathbb{E}[\lVert X_n - X \rVert^m] \rightarrow 0 \) as \( n \rightarrow \infty \), where $\mathbb{E}$ denotes the mathematical expectation with respect to $\Q$.

\section{Description of the model}
Let $\G = (V,\E)$ be a directed graph, where $V$ is a non-empty finite set of vertices and $\E\subseteq V \times V$ is a set of directed edges. Assume that for each vertex $v \in V$, the in-degree $d_v^{\text{in}} := |\{v' \in V \mid (v', v) \in \E\}| \geq 1$, meaning that every vertex has at least one in-neighbour.

Now, associate an elephant with each vertex $v$, referred to as elephant $v$. All elephants start their walk from the origin $0$. At time step $n = 1$, elephant $v$ takes its first step $\X{1}{v}$ independently of the other elephants, where $\X{1}{v}$ follows a Rad($q_v$) distribution for some $q_v\in[0, 1]$.

For time steps $n \geq 2$, elephant $v$ decides its next step $\X{n}{v}\in\{-1, 1\}$ based on the walking history of its in-neighbouring elephants. The current position of elephant $v$ at time $n$ is denoted as $\s{n}{v} \in \mathbb{Z}$. Each elephant $v$ is assigned a memory parameter $p_v \in [0, 1]$, which represents the likelihood of following the historical steps of its in-neighbouring elephants.

At each time step $n + 1$, elephant $v$ will uniformly select one of its in-neighbouring elephants $v'$ and choose a past step uniformly from the first $n$ steps of $v'$. With probability $p_v$, it will repeat that step, while with probability $1 - p_v$, it will take the opposite step. This process is carried out independently by each elephant, given that all the elephants have walked up to time $n$.

More precisely, let $\{p_v~|~v\in V\}$ and $\{q_v~|~v\in V\}$ be any collection or real numbers in $[0,1]$. Suppose \(\{\sigma^{(v)},\yi{n}{v},\D{n}{v},\us{n}{v} \mid n \geq 1, v \in V\} \) is a set of independent random variables such that:
\begin{itemize}
\item[1.] \(\sigma^{(v)}\sim\text{Rad}(q_v)\), for all $v\in V$.
\item[2.] \( \yi{n}{v} \sim \text{Rad}(p_v) \), for all \( v \in V \).
\item[2.] \( \D{n}{v} \sim \text{Unif}(\{1, \ldots, n\}) \), for all \( n \geq 1 \).
\item[3.] \( \us{n}{v} \sim \text{Unif}(\{v' \mid (v', v) \in \E \}) \), for all \( n \geq 1 \) and \(v\in V\).
\end{itemize}

Define the step sequence \( \{\X{n}{v} \mid n \geq 1, v \in V\} \) recursively as follows:\\
Set \( \X{1}{v} = \sigma^{(v)} \) for all \( v \in V \). For \( n \geq 1 \), after defining \( \{\X{j}{v} \mid j \leq n, v \in V\} \), inductively define:
\[
\X{n+1}{v} = \yi{n}{v} \X{\D{n}{v}}{\us{n}{v}}\hspace{0.5cm}\text{,for all \( v \in V \)}
\]
Set $\s{0}{v}=0$ and $\s{n}{v}=\sum\limits_{j=1}^n\X{j}{v}$ for all $v\in V$.

The joint process \( \{ \s{n}{v} \mid n \geq 0, v \in V \} \) is referred to as the \emph{elephant random walks with graph based shared memory} (or in short ERWG) with relation graph \( \G \), memory parameters \( \{ p_v \mid v \in V \} \), and initial steps \( \{ \sigma^{(v)} \mid v \in V \} \). This can be written in short as
\[
S_n:=\{\s{n}{v}~|~v\in V\} \sim \text{ERWG}(\{p_v\}_{v\in V}, \{q_v\}_{v\in v}, \G).
\]

\subsection{A general representation of ERWG:}
For simplicity assume that the vertex set $V=\{1,...,k\}$ for some integer $k\geq1$. Define for all $n\geq1$, vector of steps of the elephants: 
$$X_n=(\X{n}{1},...,\X{n}{k})$$
and for all $n\geq0$, vector of positions of the elephants:
$$S_n=(\s{n}{1},...,\s{n}{k})$$ 
Further define for all $n\geq1$ 
$$\F{n}=\sigma(S_j~|~j\leq n,v\in V)$$ 
be the sigma-algebra generated by $\{S_j~|~j\leq n\}$, in other words the history of all elephants upto time $n$. 

\begin{definition}
The $k\times k$ matrix $B$ with $(i,j)$'th entry $b_{i,j}:=\frac{2p_j-1}{d_j^{in}}\bb{1}_{\{(i,j)\in\E\}}$ for $i,j\in\{1,...,k\}$ is called the \emph{memory matrix} of the ERWG process $\{S_n~|~n\geq0\}$.
\end{definition}
First, observe that for all \( j \in V \),

\begin{equation}\label{0}
\mathbb{E}[\X{n+1}{j} \mid \F{n}] = \frac{2p_j - 1}{d_j^{\text{in}}} \sum_{l : (l, j) \in \E} \frac{\s{n}{l}}{n}, \quad \text{for all } n \geq 1,
\end{equation}
which can be reformulated as

\begin{equation}\label{aa}
\mathbb{E}[X_{n+1} \mid \F{n}] = \frac{S_n}{n} B, \quad \text{for all } n \geq 1,
\end{equation}
and by definition,

\begin{equation}\label{bb}
S_{n+1} = S_n + X_{n+1}, \quad \text{for all } n \geq 0,
\end{equation}
where \( S_0 = 0 \). It will be shown later that, under certain mild conditions on \( B \), we have

\begin{equation}\label{cc}
\text{Var}(X_{n+1} \mid \F{n}) \xrightarrow{\text{a.s.}} \Theta,
\end{equation}
where \( \Theta \) is a non-random positive definite matrix in \( \mathbb{R}^{k \times k} \).

\begin{remark}
The analysis in this paper applies to any process \( S_n \) with dynamics described by \eqref{aa}, \eqref{bb} and \eqref{cc}, whenever the matrix \( B \)  satisfies \( \sum_{i=1}^n |B_{i, j}| \leq 1 \) for each \( j = 1, \ldots, k \).
\end{remark}

\subsection{Connection with Other Network-Based Reinforced Stochastic Processes:}

Assume \( p_j = 1 \) for all \( j = 1, \ldots, k \), making \( B \) the weighted adjacency matrix of \( \G \). Define \( Y_n = \frac{1}{2}(X_{n+2} + \mathbf{1}) \), \( Z_n = \frac{S_{n+2}}{2n + 4} + \frac{1}{2} \mathbf{1} \), and \( \mathcal{G}_n = \mathcal{F}_{n+2} \) for \( n \geq 0 \). Then, for all \( n \geq 0 \), it can be verified that
\[
Z_{n+1} = (1 - r_n) Z_n + r_n Y_{n+1},
\]
where \( r_n = \frac{1}{n+3} \) and \( \mathbb{E}[Y_{n+1} \mid \mathcal{G}_n] = Z_n B \). This process \( Z_n \) corresponds to the \emph{networks of reinforced stochastic processes} studied in \cite{Synchronisation}, \cite{Functionalcltfornetwork}, \cite{Empiricalmean}, \cite{Weightedempiricalmean}, \cite{Completefirstorder}, and \cite{Polarization}.\\

Observe that \( Z_n = \mathbf{1} \) on the set \( [X_1 = \mathbf{1}] \) and \( Z_n = \mathbf{0} \) on \( [X_1 = -\mathbf{1}] \). Further, if we assume that \( \G \) is strongly connected (i.e., \( B \) is irreducible), then by Theorem 2.2 of \cite{Completefirstorder} and Theorem 3.1 of \cite{Polarization}, it follows that on the set \( \bigcup_{j=1}^k \left[ S_{j}^{(2)} = 0 \right] \), there is almost sure synchronization: \( Z_n \xrightarrow{\text{a.s.}} Z_{\infty} \mathbf{1} \) for some random variable \( Z_{\infty} \) such that \( \mathbb{P}(0 < Z_{\infty} < 1) = 1 \).\\

By setting \( S_{\infty} = 2 Z_{\infty} - 1 \), we conclude that if \( \G \) is strongly connected, then on the set \( \bigcup_{j=1}^k \left[ S_{j}^{(2)} = 0 \right] \), all walks \( S_n^{(j)} \) converge almost surely to a common random variable \( S_{\infty} \) with \( \mathbb{P}(-1 < S_{\infty} < 1) = 1 \).\\

In the final section, we explore the limit \( S_{\infty} \), in a particularly interesting special case, focusing on whether it is degenerate or has a positive variance.

\subsection{Connection to Stochastic Approximation Schemes:}
Let \( Z_n = \frac{S_n}{n} \) represent the average walk of all elephants. From \eqref{aa} and \eqref{bb}, it follows that \( Z_n \) can be expressed as
\begin{equation}\label{sa}
Z_{n+1} = Z_n - \frac{h(Z_n)}{n+1} + \frac{\Delta M_{n+1}}{n+1}
\end{equation}
where \( h: \mathbb{R}^k \rightarrow \mathbb{R}^k \) is defined by \( h(x) = x(I - B) \) for all \( x \in \mathbb{R}^k \), and \( \Delta M_{n+1} := X_{n+1} - \frac{S_n}{n} B \) forms a martingale difference sequence. Define \( \rho = \max \{ \Re(\lambda) \mid \lambda \in \text{Sp}(I - B) \} \). The asymptotic behaviour of any process \( Z_n \) satisfying \eqref{sa}, under standard assumptions on the sequence \( \Delta M_{n+1} \), is governed by the parameter \( \rho \).\\

For instance, if \( \rho < 1 \), then \( Z_n \xrightarrow{\text{a.s.}} 0 \) (see \cite{Borkar}). Specifically, if \( \rho < \frac{1}{2} \), a non-trivial CLT applies to \( \sqrt{n} Z_n \) (see, e.g., \cite{duflo2013random} and \cite{Zhang2}).\\

When \( \rho = \frac{1}{2} \), a non-trivial CLT holds for \( \sqrt{\frac{n}{\log n}} Z_n \) if the Jacobian \( Dh(0) = I - B \) is scalar, i.e., \( B = \frac{1}{2} I \) (see Theorem 2.2.12 in \cite{duflo2013random}). Furthermore, if \( \rho > \frac{1}{2} \) and \( Dh(0) = \rho I \), then \( n^{\rho} Z_n \) converges almost surely to a real-valued random variable.\\

The case \( \rho \geq \frac{1}{2} \) for a general (not necessarily scalar) Jacobian \( Dh(0) \) is examined in \cite{Zhang2}. It was shown that a CLT holds for \( \frac{\sqrt{n}}{(\log n)^{\nu - \frac{1}{2}}} Z_n \) when \( \rho = \frac{1}{2} \) (see Theorem 2.1 of \cite{Zhang2}), where \( \nu \) represents the maximum of the algebraic multiplicities of all eigenvalues of \( Dh(0) \) with real part \( \frac{1}{2} \).\\

Notably, if \( Dh(0) = \frac{1}{2} I \), then \( \nu = k \), making Theorem 2.1 of \cite{Zhang2} trivial, as in this case, $\sqrt{\frac{n}{\log n}} Z_n$ converges weakly to a normal vector (Theorem 2.2.12 of \cite{duflo2013random}).\\

This raises the question: how can we establish a rate for the CLT when \( \rho = \frac{1}{2} \) for a general Jacobian \( Dh(0) \) that matches the rate given in Theorem 2.2.12 of \cite{duflo2013random}? Or, to what extent can we relax the scalar assumption on \( Dh(0) \) while preserving the rate \( \sqrt{\frac{n}{\log n}} \) for the CLT of \( Z_n \)?\\

In this paper, we establish a strong invariance principle for \( Z_n \) (see Theorem 3) and demonstrate that the rate \( \sqrt{\frac{n}{\log n}} \) remains valid when the Jacobian \( Dh(0) \) is diagonalizable in \( \mathbb{C}^{k \times k} \) (see Theorem 4). We further refine these results when \( B \) is diagonalizable in \( \mathbb{R}^{k \times k} \).\\

Our methods are easily applicable to general stochastic approximation schemes such as:
\[
Z_{n+1} = Z_n - \frac{h(Z_n)}{n+1} + \frac{\Delta M_{n+1} + r_{n+1}}{n+1}
\]
where the function $h$, martingale difference sequence $\Delta M_{n+1}$ and the error sequence $r_{n+1}$ satisfy the conditions stated in \cite{Zhang2}.\\

As a result, we obtain strong invariance for the classical elephant random walk (see Corollary 2) with better rates than \cite{InvarianceERW1} and \cite{InvarianceIERW2}.

\subsection{Connection to Another Statistical Physics Model:}  
In 2019, Marquioni introduced the \emph{Multidimensional Elephant Random Walk with Coupled Memory} model in \cite{erwcoupled}. It can be verified that the dynamics of his general joint process (see (44), (45), and (46) in \cite{erwcoupled}) satisfy the conditions \eqref{aa}, \eqref{bb} and \eqref{cc} for a matrix \( B = ((b_{i,j}))_{1 \leq i,j \leq k} \) with entries in the interval \([-1,1]\).\\

Since our analysis here is valid as long as \( \sum_{i=1}^k |b_{i,j}| \leq 1 \) for each \( j = 1, \dots, k \), our results can be used to derive first- and second-order asymptotic results for Marquioni’s joint process as a direct consequence.\\

It is also noteworthy that Marquioni calculated asymptotic orders for the moments in specific cases, including the \emph{two-elephant model} and the \emph{cow-and-ox model}. Furthermore, he employed the Fokker-Planck equation, using characteristic functions, to obtain a continuous-time approximation for the process.

\section{Main results for general ERWG}
Let us recall the general representation of the ERWG process \( S_n \) described in Section 2.1. Note from \eqref{aa} and \eqref{bb} we can write $S_n$ as the following recursive relation:

\begin{equation} \label{1}
S_{n+1} = S_n \left(I + \frac{B}{n}\right) + \Delta M_{n+1} \hspace{0.25cm} \text{for all}~n \geq 1,
\end{equation}

where \( B \) is the memory matrix of the ERWG process \( S_n \), and \( \Delta M_{n+1} \) is a mean-zero martingale difference sequence in \( \mathbb{R}^k \), adapted to the filtration \( \{\mathcal{F}_n\}_{n \geq 0} \), defined by

\[
\Delta M_{n+1} = X_{n+1} - \frac{S_n}{n} B \hspace{0.25cm} \text{for all}~n \geq 1.
\]

The asymptotic behaviour of the sequence \( S_n \) depends on the eigen-structure of \( B \) and, in certain cases, on the initial configuration probabilities \( \{q_v\}_{v \in V} \) as well.

Suppose \( \lambda_1, \ldots, \lambda_k \) are the eigenvalues of \( B \). We define 

\begin{equation} \label{-1}
\eta = \max \{\Re(\lambda_j) \mid j = 1, \dots, k\}.
\end{equation}

The following lemma confirms that all eigenvalues of \( B \) lie within the unit disk in \( \mathbb{C} \). 

\begin{lemma}
For all \( j = 1, \dots, k \), \( |\lambda_j| \leq 1 \).
\end{lemma}

\begin{proof}
Fix \( j \in \{1, \dots, k\} \) and take \( x = (x_1, \dots, x_k) \in \mathbb{C}^k \setminus \{0\} \) such that \( xB = \lambda_j x \). Let \( j_0 \in \{1, \dots, k\} \) be chosen so that \( \max_{1 \leq j \leq k} |x_j| = |x_{j_0}| \). Then observe that

\begin{align*}
&\lambda_j x_{j_0} = \frac{2p_{j_0 - 1}}{d_{j_0}^{\text{in}}} \sum_{l: (l, j_0) \in \mathcal{E}} x_l \\
\implies &|\lambda_j| |x_{j_0}| \leq |2p_{j_0} - 1| |x_{j_0}| \sum_{l: (l, j_0) \in \mathcal{E}} \frac{1}{d_{j_0}^{\text{in}}} = |2p_{j_0} - 1| |x_{j_0}| \leq |x_{j_0}| \\
\implies &|\lambda_j| \leq 1.
\end{align*}
\end{proof}

This lemma ensures that \( \eta \in [-1, 1] \). The parameter \( \eta \) functions as a global tuning parameter, as we shall see that the overall asymptotic properties of \( S_n \) are determined by this value.

\begin{definition}
We say that the ERWG process \( S_n \) is in a globally diffusive regime if \( \eta < \frac{1}{2} \), in a globally critical regime if \( \eta = \frac{1}{2} \), and in a globally super-diffusive regime if \( \eta > \frac{1}{2} \).
\end{definition}

\subsection{When the memory matrix $\mathbf{B}$ is diagonalisable in $\mathbf{\mathbb{C}^{k\times k}}$:} 
Suppose there is an invertible $T\in\mathbb{C}^{k\times k}$ such that
$$T^{-1}BT=\text{diag}(\lambda_1,...,\lambda_k)=:\Lambda$$
In case if $B$ is symmetric, we assume that $T$ is orthogonal, that is $T'T=I$. 
\begin{definition}
We say that the ERWG $S_n$ is diagonalisable (respectively symmetric) if its memory matrix $B$ is diagonalisable (respectively symmetric).
\end{definition}
\subsubsection{\textbf{Projected walks and projection martingales for diagonalisable ERWG:}}
Define the vector of projected walks $\hat{S}_n=S_nT$ along with the projected steps $\hat{X}_n=X_nT$. The corresponding martingale difference sequence is $\Delta\hat{M}_{n+1}=\Delta M_{n+1}T=\hat{X}_{n+1}-\frac{\hat{S}_n}{n}\Lambda$.

\begin{definition}
A projected walk $\st{n}{j}=S_nTe_j'$ is called diffusive, critical or super-diffusive if $\Re(\lambda_j)$ is less than, equal to or greater than $\frac{1}{2}$ respectively.
\end{definition}
Notice that if $S_n$ is in globally diffusive regime then all the projected walks are diffusive.\\\\
Post-multiplying both sides of \eqref{1} by $T$ and extracting the identity corresponding to the $j$'th component, we get for $n\geq2$

\begin{align*}
&\st{n+1}{j}=\left(1+\frac{\lambda_j}{n}\right)\st{n}{j}+\Delta\hat{M}_{n+1}^{(j)}\\
\implies &\st{n+1}{j}=d_{n+1}^{(j)}\big[\st{2}{j}+(d_3^{(j)})^{-1}\Delta\hat{M}_3^{(j)}+...+(d_{n+1}^{(j)})^{-1}\Delta\hat{M}_{n+1}^{(j)}\big]
\end{align*}
where $d_{n+1}^{(j)}=\prod\limits_{l=2}^n\left(1+\frac{\lambda_j}{l}\right)=\frac{\Gamma(\lambda_j+n+1)}{\Gamma(\lambda_j+2)\Gamma(n+1)}$. Hence we can write
\begin{equation}\label{2}
\st{n}{j}=d_n^{(j)}\big[\st{2}{j}+L_n^{(j)}\big],\hspace{0.25cm}\forall n\geq3,~j=1,...,k
\end{equation}
where 
\begin{equation}
L_n^{(j)}=\sum\limits_{l=3}^n(d_l^{(j)})^{-1}\Delta\hat{M}_l^{(j)},\hspace{0.25cm}\forall n\geq3,j=1,...,k
\end{equation}
Notice that $L_n^{(j)}$ is a mean 0 square integrable complex martingale adapted to the filtration $\F{n}$. 

\begin{definition}
The complex martingales $(d_n^{(j)})^{-1}\st{n}{j}=L_n^{(j)}+\st{2}{j}$ are called the projection martingales associated with the diagonalisable ERWG $S_n$.
\end{definition}

\begin{theorem}[\textbf{Strong asymptotic behaviour of projected walks for diagonalisable ERWG}] For all $j=1,...,k$ and for all $\epsilon>0$, we have
\begin{equation}\label{t1-1}
\st{n}{j}=\begin{cases}
o\big(n^{\frac{1}{2}}(\log n)^{\frac{1}{2}+\epsilon}\big)~a.s. &\text{if}~\Re(\lambda_j)<\frac{1}{2}~\text{(diffusive)}\\[0.25cm]
o\big((n\log n)^{\frac{1}{2}}(\log \log n)^{\frac{1}{2}+\epsilon}\big)~a.s. &\text{if}~\Re(\lambda_j)=\frac{1}{2}~(critical)\end{cases}
\end{equation}
and for any $m>0$
\begin{equation}
\label{t12}
\mathbb{E}[|\st{n}{j}|^m]=\begin{cases}
O(n^{\frac{m}{2}}), &\text{if $\Re(\lambda_j)<\frac{1}{2}$}~(diffusive)\\[0.25cm]
O((n\log n)^{\frac{m}{2}}), &\text{if $\Re(\lambda_j)=\frac{1}{2}$}~(critical)
\end{cases}
\end{equation}
Finally if $\Re(\lambda_j)>\frac{1}{2}$~(super-diffusive) then 
\begin{equation}
\label{t11}
(d_n^{(j)})^{-1}\st{n}{j}\xrightarrow{a.s.,~L^m}\st{\infty}{j}\hspace{0.25cm}\text{for all $m>0$}
\end{equation}
For some complex random variable $\st{\infty}{j}$  with $\mathbb{E}[\st{\infty}{j}]=\mathbb{E}[\st{2}{j}]$. In particular if $\eta<1$ then $\mathbb{E}\big[\big|\st{\infty}{j}-\mathbb{E}[\st{\infty}{j}]\big|^2\big]>0$. 
\end{theorem}
\vspace{1cm}
Following Corollary is an immediate consequence of \eqref{t1-1} and \eqref{t11} of Theorem 1 and we state it without proof.

\begin{corollary}[\textbf{Strong asymptotic behaviour of average walk for diagonalisable ERWG}]
\begin{equation}\label{9}
\frac{S_n}{n}\xrightarrow{a.s.}ST^{-1}
\end{equation}
where the $j$'th component of $S$ is $S^{(j)}:=\st{\infty}{j}\bb{1}_{\{\lambda_j=1\}}$. In particular if $\eta<1$ then $S=0$. Further If $\eta\leq\frac{1}{2}$ then for all $\epsilon>0$ 
\begin{equation}\label{10}
\frac{S_n}{n}=o\big(n^{-\frac{1}{2}}(\log n)^{\frac{1}{2}+\epsilon}\big)~~a.s.
\end{equation}\\
\end{corollary}

\begin{theorem}[\textbf{Weak convergence and law of iterated logarithm for projected walks for diagonalisable ERWG}]
Assume that the memory matrix $B$ is diagonalisable in $\mathbb{R}^{k\times k}$ and $\eta<1$. Define 
\begin{equation*}
\sigma^2(\lambda)=
\begin{cases}
(1-2\lambda)^{-1}e_jT'Te_j' &\text{,if~}\lambda<\frac{1}{2}\\
 e_jT'Te_j' &\text{if~}\lambda=\frac{1}{2}\\
 (2\lambda-1)^{-1}\Gamma(\lambda+2)^2e_jT'Te_j' &\text{if~}\lambda>\frac{1}{2}
 \end{cases}
 \end{equation*}
Then for all $j=1,...,k$ we have the following:\\

If $\lambda_j<\frac{1}{2}$ ,i.e. $\st{n}{j}$ is diffusive then 
\begin{equation}
\label{t21}
\frac{\st{n}{j}}{\sqrt{n}}\xrightarrow{d}\mathcal{N}(0,\sigma^2(\lambda_j))
\end{equation}
and with probability 1, the sequence $\frac{\st{n}{j}}{\sqrt{2n\log\log n}}$ is relatively compact with set of limit points $[-\sigma(\lambda_j),\sigma(\lambda_j)]$.\\

If $\lambda_j=\frac{1}{2}$ ,i.e. $\st{n}{j}$ is critical then 
\begin{equation}
\label{t22}
\frac{\st{n}{j}}{\sqrt{n\log n}}\xrightarrow{d}\mathcal{N}(0,\sigma^2(1/2))
\end{equation}
and with probability 1, the sequence $\frac{\st{n}{j}}{\sqrt{2n\log n\log\log\log n}}$ is relatively compact with set of limit points $[-\sigma(1/2),\sigma(1/2)]$.\\

Finally if $\frac{1}{2}<\lambda_j<1$ ,i.e. $\st{n}{j}$ is super-diffusive then the almost sure limit $\st{\infty}{j}$ (defined in Theorem 1) satisfies $\mathbb{E}[\st{\infty}{j}]=\mathbb{E}[\st{2}{j}]$ and $\emph{Var}(\st{\infty}{j})>0$.\\

Further we have
\begin{equation}
\label{t23}
n^{\lambda_j-\frac{1}{2}}\big((d_n^{(j)})^{-1}\st{n}{j}-\st{\infty}{j}\big)\xrightarrow{d}\mathcal{N}(0,\sigma^2(\lambda_j))
\end{equation}
and that, with probability 1, the sequence $n^{\lambda_j-\frac{1}{2}}(2\log\log n)^{-\frac{1}{2}}\big((d_n^{(j)})^{-1}\st{n}{j}-\st{\infty}{j}\big)$ is relatively compact with set of limit points $[-\sigma(\lambda_j),\sigma(\lambda_j)]$
\end{theorem}

\begin{theorem}[\textbf{Strong Gaussian Approximation of diagonalisable ERWG in globally diffusive and critical regime}] 
If $\eta\leq\frac{1}{2}$ then there exists a probability space where the entire sequence $\{S_n\}_{n\geq0}$ is redefined without changing it's distribution and in the same space there exists a sequence $\{Y_n\}_{n\geq1}$ of i.i.d. $\mathcal{N}(0,I)$ random vectors such that for some $\delta>0$
\begin{equation}\label{18}
S_n=\begin{cases}
\sum\limits_{j=1}^n Y_j\big(\frac{j}{n}\big)^{-B}+O\big(n^{\frac{1}{2}-\delta}\big)~~a.s. &\text{,if}~\eta<\frac{1}{2}\\[0.25cm]
\sum\limits_{j=1}^n Y_j\big(\frac{j}{n}\big)^{-B}+O\big(n^{\frac{1}{2}}\big)~~a.s. &\text{,if}~\eta=\frac{1}{2}
\end{cases}
\end{equation}
Further if $B$ is diagonalisable in $R^{k\times k}$ then 
\begin{equation}\label{d}
S_n=\sum\limits_{j=1}^n Y_j\bigg(\frac{j}{n}\bigg)^{-B}+\epsilon_nT^{-1}~a.s.
\end{equation}
where $\epsilon_n^{(j)}=O(n^{\frac{1}{2}-\delta_j})$ for all $j=1,...,k$ with $\delta_j>0$ if $\lambda_j<\frac{1}{2}$ and $\delta_j=0$ if $\lambda_j=\frac{1}{2}$.\\
\end{theorem}

\begin{theorem}[\textbf{Joint asymptotic normality of diagonalisable ERWG in globally diffusive and critical regime}]
If $\eta<\frac{1}{2}$ then 
\begin{equation}\label{clt1}
\frac{S_n}{\sqrt{n}}\xrightarrow{d}\mathcal{N}_k(0,\Sigma^{(1)}),\hspace{0.25cm}\text{where}
\end{equation} 
\begin{equation}\label{s1}\hspace{0.25cm}\Sigma^{(1)}=\int_0^{\infty}e^{(B-\frac{1}{2}I)'t}e^{(B-\frac{1}{2}I)t}dt
\end{equation}
and if $\eta=\frac{1}{2}$ then
\begin{equation}\label{clt2}
\frac{S_n}{\sqrt{n\log n}}\xrightarrow{d}\mathcal{N}_k(0,\Sigma^{(2)}),\hspace{0.25cm}\text{where}
\end{equation} 
\begin{equation}\label{s2}
\hspace{0.25cm}\Sigma^{(2)}=\lim\limits_{n\rightarrow\infty}\frac{1}{\log n}\int_0^{\log n}e^{(B-\frac{1}{2}I)'t}e^{(B-\frac{1}{2}I)t}dt
\end{equation}
In particular if $B$ is symmetric then $\Sigma^{(1)}=(I-2B)^{-1}$ is defined for $\eta<\frac{1}{2}$ and $\Sigma^{(2)}=\sum\limits_{j=1}^k\bb{1}_{\{\lambda_j=\frac{1}{2}\}}Te_j'e_jT'$ is defined for $\eta=\frac{1}{2}$
\end{theorem}

\begin{remark}
Recall the stochastic approximation algorithm \eqref{sa} from Section 2.3 for \( Z_n = \frac{S_n}{n} \). By definition \( \rho = 1 - \eta \). Suppose the Jacobian \( Dh(0) \) is diagonalizable in \( \mathbb{C}^{k \times k} \) and \( \rho = \frac{1}{2} \). Then, by Theorem 4, we conclude that \( \sqrt{\frac{n}{\log n}} Z_n \) follows a central limit theorem (CLT) with limiting covariance matrix \( \Sigma^{(2)} \). It is easy to see that \( \Sigma^{(2)} \) has at least one non-zero entry in this scenario. However, Theorem 2.1 from \cite{Zhang2} states that \( \frac{\sqrt{n}}{(\log n)^{\nu - \frac{1}{2}}} Z_n \) also satisfies a CLT with a limiting covariance matrix \( \tilde{\Sigma} \). Recall that \( \nu \) is defined as the maximum of the algebraic multiplicities of all eigenvalues of \( Dh(0) \) with real part \( \rho = \frac{1}{2} \). This rate aligns with our result if and only if \( \nu = 1 \), meaning there is a unique eigenvalue with real part \( \rho = \frac{1}{2} \). Clearly, if \( Dh(0) \) is diagonalizable but has multiple eigenvalues with real part \( \frac{1}{2} \), then \( \tilde{\Sigma} \) becomes the zero matrix.
\end{remark}

\begin{theorem}[\textbf{Joint asymptotic normality of projected walks for symmetric ERWG in globally critical regime}]
Assume that $B$ is symmetric and $\eta=\frac{1}{2}$. Let $a_n^{(j)}=n^{-\frac{1}{2}}$ or $(n\log n)^{-\frac{1}{2}}$ according as $\lambda_j<\frac{1}{2}$ or $\lambda_j=\frac{1}{2}$ respectively. Further let $a^{(j)}=\frac{1}{1-2\lambda_j}$ or $1$ according as $\lambda_j<\frac{1}{2}$ or $\lambda_j=\frac{1}{2}$ respectively. Then following CLT holds 
\begin{equation}\label{clt3}
S_nT\emph{diag}\big(a_n^{(1)},...,a_n^{(k)}\big)\xrightarrow{d}\mathcal{N}_k(0,\emph{diag}(a^{(1)},...,a^{(k)}))
\end{equation}\\
\end{theorem}

\begin{theorem}[\textbf{Approximation of diagonalisable ERWG by linear combination of Brownian motions in globally diffusive and critical regime}]
Assume that $B$ is diagonalisable in $\mathbb{R}^{k\times k}$ and $T=((t_{i,j}))_{1\leq i,j\leq k}$. Suppose $\eta\leq\frac{1}{2}$. Then there exist (possibly in some different probability space where the entire ERWG process $S_n$ is redefined without changing it's distribution) independent 1-dimensional Brownian motions $\{W^{(1)}(s)\}_{s\geq0},...,\{W^{(k)}(s)\}_{s\geq0}$ such that 
\begin{equation}\label{emb1}
S_nT=G_n+\epsilon_n\hspace{0.25cm}a.s.
\end{equation}
where the $p$'th coordinate of $G_n$ is $G_n^{(p)}:=n^{\lambda_p}\sum\limits_{l=1}^kt_{l,p}W^{(l)}\left(\sum\limits_{j=1}^nj^{-2\lambda_p}\right)$ for all $p=1,...,k$ and $\epsilon_n$ is defined as in Theorem 3.\\
\end{theorem}

\begin{corollary}[\textbf{Approximation of symmetric ERWG by Brownian motions in globally diffusive and critical regime}]
If $B$ is symmetric and $\eta\leq\frac{1}{2}$, then there exist (possibly in some different probability space where the entire ERWG process $S_n$ is redefined without changing it's distribution) independent Brownian motions $\{W^{(1)}(s)\}_{s\geq0},...,\{W^{(k)}(s)\}_{s\geq0}$ such that 
\begin{equation}\label{c21}
S_nT=\left(n^{\lambda_1}W^{(1)}\left(\sum\limits_{j=1}^nj^{-2\lambda_1}\right),...,n^{\lambda_k}W^{(k)}\left(\sum\limits_{j=1}^nj^{-2\lambda_k}\right)\right)+\epsilon_n~a.s.
\end{equation}
where $\epsilon_n$ is defined as in Theorem 3.
\end{corollary}

\begin{remark}
Observe that if \( \G \) consists of a single vertex with a directed self-loop in it, specifically, \( V = \{1\} \) and \( \E = \{(1,1)\} \), then \( B = 2p_1 - 1 \) and \( T = 1 \). In this case, the process \( S_n \) corresponds to the classic elephant random walk with memory parameter \( p_1 \). From Corollary 2, we conclude that \( S_n = n^{2p_1 - 1} W\left(\sum_{j=1}^n j^{2 - 4p_1}\right) + \epsilon_n \) $a.s.$, where \( \epsilon_n = O(n^{\frac{1}{2} - \delta}) \) for some \( \delta > 0 \) if \( 2p_1 - 1 < \frac{1}{2} \), and \( \epsilon_n = O(n^{\frac{1}{2}}) \) if \( 2p_1 - 1 = \frac{1}{2} \). These rates improve upon those established in \cite{InvarianceERW1} and \cite{InvarianceIERW2}.
\end{remark}

\subsubsection{\textbf{Joint asymptotic behaviour of diffusive and critical projected walks:}}
Till now we have seen the strong and weak asymptotic behaviour of the joint process $S_n$ in globally diffusive and critical regime. Note that if the ERWG $S_n$ is in globally super-diffusive regime (i.e. $\eta>\frac{1}{2}$) then there must be atleast one super-diffusive projected walk. The marginal asymptotic behaviour of super-diffusive projected walks are described in Theorem 2. Note that a globally super-diffusive $S_n$ can still contain some diffusive or critical projected walks. in this section we will study the joint asymptotic behaviour of such projected walks.\\ 

Suppose $\eta<1$ (not necessarily greater than $\frac{1}{2}$) and there exists $j\in\{1,...,k\}$ such that the projected walk $\st{n}{j}=S_nTe_j'$ is not super-diffusive, i.e. $\Re(\lambda_j)\leq\frac{1}{2}$.\\

Suppose $k_1:=|\{j~|~\Re(\lambda_j)<\frac{1}{2}\}|$ and $k_2:=|\{j~|~\Re(\lambda_j)=\frac{1}{2}\}|$. We assume that $\tk:=k_1+k_2\geq1$.\\

For simplicity we shall stick to the case when $B$ is diagonalisable in $\mathbb{R}^{k\times k}$ and $k\geq2$. That is $T^{-1}BT=\Lambda:=\text{diag}(\lambda_1,...,\lambda_k)$ for some real invertible matrix $T$ and real eigenvalues $\lambda_1,...,\lambda_k$. Without loss assume that $\lambda_1,...,\lambda_{\tk}$ are the only eigenvalues which are less or equal to $\frac{1}{2}$. If $k_1\geq1$ then $\lambda_1,...,\lambda_{k_1}$ are the only eigenvalues which are less than $\frac{1}{2}$ and if $k_1=0$ then all of $\lambda_1,...,\lambda_{\tk}$ are $\frac{1}{2}$. Our goal is to study the joint asymptotic behaviour of the projected walks $\hat{S}_n^{(1)}=S_nTe_1',...,\hat{S}_n^{(\tk)}=S_nTe_{\tk}'$. Notice that if $\eta>\frac{1}{2}$ then for $\tk+1\leq j\leq k$ the super-diffusive projected walks $\hat{S}_n^{(j)}=S_nTe_j'$ have an almost sure limit with scaling $d_n^{(j)}$ by Theorem 1.\\

Define $\ts_n=(\hat{S}_n^{(1)},...,\hat{S}_n^{(\tk)}),~\tx_n=(\hat{X}_n^{(1)},...,\hat{X}_n^{(\tk)}),~\tm_{n+1}=(\Delta\hat{M}_{n+1}^{(1)},...,\Delta\hat{M}_{n+1}^{(\tk)})=\tx_{n+1}-\frac{\ts_n}{n}\tl$ where $\tl=\text{diag}(\lambda_1,...,\lambda_{\tk})$. Let $\ti=[e_1',...,e_{\tk}']$, that is $\ti$ is obtained by deleting last $k-\tk$ columns of the identity matrix $I_k$. \\

Then clearly $\ts_n=S_nT\ti,~\tx_n=X_nT\ti$ and $\tl=\ti'\Lambda\ti$. Further note that 
\begin{align*}
\tm_{n+1}&=X_{n+1}T\ti-\frac{S_n}{n}T\ti\ti'\Lambda\ti\\
&=X_{n+1}T\ti-\frac{S_n}{n}T\ti\ti'\Lambda\ti-\frac{S_n}{n}T(I_k-\ti\ti')\Lambda\ti\\
&=X_{n+1}T\ti-\frac{S_n}{n}T\Lambda\ti\\
&=X_{n+1}T\ti-\frac{S_n}{n}BT\ti\\
&=\Delta M_{n+1}T\ti
\end{align*} 
With these definitions, we have the following Theorem similar to Theorem 2:

\begin{theorem}[\textbf{Strong Gaussian approximation of diffusive and critical projected walks for diagonalisable ERWG}]
Suppose $B$ is diagonalisable in $\mathbb{R}^{k\times k}$ such that $\eta<1$ and $\lambda_j\leq\frac{1}{2}$ for some $j$. Then there exists a probability space where the entire ERWG process $S_n$ is redefined without changing it's distribution along with i.i.d. $\mathcal{N}_{\tk}(0,I_{\tk})$ random vectors $Y_1,...,Y_n,...$ such that
\begin{equation}\label{d2}
\ts_n=\sum\limits_{j=1}^{n}Y_j(\ti'T'T\ti)^{\frac{1}{2}}\left(\frac{j}{n}\right)^{-\tl}+\tilde{\epsilon}_n~a.s.
\end{equation}
where $\tilde{\epsilon}_n^{(j)}=O(n^{\frac{1}{2}-\delta_j})$ for $j=1,...,\tk$, where $\delta_j>0$ if $j\leq k_1$ and $\delta_j=0$ if $k_1+1\leq j\leq\tk$.
\end{theorem}

\begin{corollary}[\textbf{Joint asymptotic normality of diffusive and critical projected walks for diagonalisable ERWG}]
Suppose $B$ is diagonalisable in $\mathbb{R}^{k\times k}$ such that $\eta<1$ and $\lambda_j\leq\frac{1}{2}$ for some $j$, then
\begin{equation}\label{co3}
\frac{\ts_n}{\sqrt{n}}\xrightarrow{d}\mathcal{N}_{\tk}(0,\tilde{\Sigma}^{(1)}),\hspace{0.25cm}\text{if}~\max\limits_{j\leq\tk}\lambda_j<\frac{1}{2}
\end{equation}
where $\tilde{\Sigma}_{p,q}^{(1)}=(1-\lambda_p-\lambda_q)^{-1}(\ti'T'T\ti)_{p,q}$ for $1\leq p,q\leq\tk$ and
\begin{equation}\label{co31}
\frac{\ts_n}{\sqrt{n\log n}}\xrightarrow{d}\mathcal{N}_{\tk}(0,\tilde{\Sigma}^{(2)}),\hspace{0.25cm}\text{if}~\max\limits_{j\leq\tk}\lambda_j=\frac{1}{2}
\end{equation}
where $\tilde{\Sigma}^{(2)}_{p,q}=\bb{1}_{\{\lambda_p+\lambda_q=1\}}(\ti'T'T\ti)_{p,q}$ for $1\leq p,q\leq\tk$.\\
\end{corollary}

Now let us focus on the diffusive projected walks only. Assume that $k_1\geq1$ and define $\ts_{n,1}=(\st{n}{1},...,\st{n}{k_1})$. Note that $\ts_{n,1}$ contains only the diffusive projected walks. Next we shall see that even if $S_n$ contains some critical projected walks, i.e. $\max\limits_{j\leq\tk}=\frac{1}{2}$, $\ts_{n,1}$ tends to show asymptotic normality with a non-singular covariance matrix.\\

\begin{corollary}[\textbf{Joint asymptotic normality of diffusive projected walks for diagonalisable ERWG}]
Suppose $B$ is diagonalisable in $\mathbb{R}^{k\times k}$ such that $\eta<1$ and $\lambda_j\leq\frac{1}{2}$ for some $j$. Then 
\begin{equation}\label{cor41}
\frac{\ts_{n,1}}{\sqrt{n}}\xrightarrow{d}\mathcal{N}_{k_1}(0,\tilde{\Sigma}_{1}^*)
\end{equation}
Where $(\tilde{\Sigma}_{1}^*)_{p,q}=(1-\lambda_p-\lambda_q)^{-1}\big(\ti_{k_1}^{'}T'T\ti_{k_1}\big)_{p,q}$ for all $1\leq p,q\leq k_1$ and the $k\times k_1$ matrix $\ti_{k_1}$ is obtained by deleting the last $k-k_1$ columns of $I_k$.\\
\end{corollary}
 
Next we establish a law of iterated logarithm for $\ts_{n,1}$.\\

\begin{theorem}[\textbf{Joint law of iterated logarithm of diffusive projected walks for diagonalisable ERWG}]
Suppose $B$ is diagonalisable in $\mathbb{R}^{k\times k}$ such that $\eta<1$ and $\lambda_j<\frac{1}{2}$ for some $j$. Then with probability 1 the sequence $\frac{\ts_{n,1}}{\sqrt{2n\log\log n}}$ is relatively compact and its set of limit points is the ellipsoid $C:=\big\{x\in\mathbb{R}^{k_1}~|~x(\tilde{\Sigma}_{1}^*)^{-1}x'\leq1\big\}$. Where $\tilde{\Sigma}_{1}^*$ is defined in Corollary 4.\\
\end{theorem}

Now let us focus on the almost sure asymptotic behaviour of the critical projected walks. Assume that $\lambda_j=\frac{1}{2}$ for some $j$, that is $k_2\geq1$. Then notice that $\lambda_{j}=\frac{1}{2}$ for all $k_1+1\leq j\leq k_1+k_2$. Define $\ts_{n,2}=(\hat{S}_n^{(k_1+1)},...,\hat{S}_n^{(k_1+k_2)})$, which is nothing but the vector of all critical projected walks. Then we have the following result:\\

\begin{theorem}[\textbf{Joint law of iterated logarithm of critical projected walks for diagonalisable ERWG}]
Suppose $B$ is diagonalisable in $\mathbb{R}^{k\times k}$ such that $\eta<1$ and $\lambda_j=\frac{1}{2}$ for some $j$. Then with probability 1 the sequence $\frac{\ts_{n,2}}{\sqrt{2n\log n\log\log \log n}}$ is relatively compact and its set of limit points is the ellipsoid $D:=\left\{x\in\mathbb{R}^{k_2}~|~x(\ti_{k_2}^{'}T'T\ti_{k_2})^{-1}x'\leq1\right\}$, where $\ti_{k_2}$ is the $k\times k_2$ matrix obtained by deleting first $k_1$ and last $k-k_1-k_2$ columns of $I_k$.\\
\end{theorem}

\subsection{When $\mathbf{B}$ is not necessarily diagonalisable:} In this section we focus on the case when the ERWG $S_n$ is not necessarily diagonalisable, i.e., $B$ is not necessarily diagonalisable. We will use stochastic approximation techniques to explore some strong and weak asymptotic behaviour of $S_n$. Recall that $\lambda_1,...,\lambda_k$ are the eigenvalues of $B$ and $\eta=\max\{\Re(\lambda_j)~|~j=1,...,k\}$. 

Let the average walk $Z_n=\frac{S_n}{n}$. First we write \eqref{1} as the following stochastic approximation scheme for the sequence $Z_n$:
\begin{equation}
\label{2.2.1}
Z_{n+1}=Z_n-\frac{h(Z_n)}{n+1}+\frac{\Delta M_{n+1}}{n+1}
\end{equation}
where the function $h:\mathbb{R}^{k}\rightarrow\mathbb{R}^k$ is defined by $h(x)=x(I-B)$ for all $x\in\mathbb{R}^k$.\\

\begin{theorem}[\textbf{A SLLN for general ERWG}]
For any memory matrix $B$ (not necessarily diagonalisable), we have
\begin{equation}
\label{t91}
S_n=o(n^{\frac{1}{2}+\delta})\hspace{0.25cm}a.s.~\text{for all $\delta>\frac{1}{2}-\frac{1}{2}\wedge(1-\eta)$}
\end{equation}\\
\end{theorem}

\begin{theorem}[\textbf{Joint law of iterated logarithm of general ERWG in globally diffusive regime}]
Suppose $B$ is any memory matrix (not necessarily diagonalisable) and $\eta<\frac{1}{2}$. Then with probability 1, the sequence $\frac{S_n}{\sqrt{2n\log\log n}}$ is relatively compact and its set of limit points is the ellipsoid $E:=\big\{x\in\mathbb{R}^k~|~x(\Sigma^{(1)})^{-1}x'\leq1\big\}$, where $\Sigma^{(1)}$ is defined in \eqref{s1}\\
\end{theorem}

Before moving to the weak convergence and strong approximation, first we need to fix some terminologies. Let $\h:=Dh(0)=I-B$ be the Jacobian of $h$. Without loss we assume that $Sp(\h)=\{1-\lambda_1,...,1-\lambda_s\}$ for some $s\leq k$, that is $1-\lambda_1,...,1-\lambda_s$ are the only distinct eigenvalues of $\h$. Suppose the Jordan canonical form of $\h$ is given by
\begin{equation}
\label{2.2.2}
\T^{-1}\h\T=\text{diag}(\J_1,...,\J_s)
\end{equation}  
where $\J_t$ is the Jordan block of order $\nu_t\times\nu_t$ corresponding to the eigenvalue $\lambda_t$, for $t=1,...,s$. Define $\rho=\min\{\Re(\lambda)~|~\lambda\in Sp(\h)\}$. Clearly $\rho=1-\eta$. We further define 
$$\nu=\max\{\nu_t~|~\Re(\lambda_t)=1-\eta\}$$ With these we are ready to present our next results:\\

\begin{theorem}[\textbf{Joint asymptotic normality of general ERWG in globally diffusive and critical regime}]
Suppose $B$ is any memory matrix (not necessarily diagonalisable).\\
 If $\eta<\frac{1}{2}$ then
\begin{equation}
\label{t111}
\frac{S_n}{\sqrt{n}}\xrightarrow{d}\mathcal{N}_k(0,\Sigma^{(1)})
\end{equation}
where $\Sigma^{(1)}$ is defined in \eqref{s1} and\\ 
 if $\eta=\frac{1}{2}$ then
\begin{equation}
\label{t112}
\frac{S_n}{\sqrt{n}(\log n)^{\nu-\frac{1}{2}}}\xrightarrow{d}\mathcal{N}_k(0,\Sigma^{(3)})
\end{equation}
where $\Sigma^{(3)}:=\lim\limits_{n\rightarrow\infty}\frac{1}{(\log n)^{2\nu-1}}\int_0^{\log n}e^{(B-\frac{1}{2}I)'t}e^{(B-\frac{1}{2}I)t}dt$\\
\end{theorem}

\begin{theorem}[\textbf{Strong Gaussian approximation of general ERWG in globally diffusive and critical regime}]
Suppose $B$ is any memory matrix (not necessarily diagonalisable) and $\eta\leq\frac{1}{2}$. Then there exists a probability space where the entire sequence $\{S_n\}_{n\geq0}$ is redefined without changing it's distribution and in the same space there exists a $k$-dimensional Brownian motion $\{B(s)\}_{s\geq0}$ such that
\begin{equation}
\label{t121}
S_n=\begin{cases}
F(n)+o(n^{\frac{1}{2}-\delta}), &a.s.,~\text{for some $\delta>0,$}~~\text{when $\eta<\frac{1}{2}$}\\[0.25cm]
F(n)+O(\sqrt{n}(\log n)^{\nu-1}), &a.s.,~\text{when $\eta=\frac{1}{2}$}
\end{cases}
\end{equation}
where $F(n):=\int_1^n dB(s)\big(\frac{s}{n}\big)^{-B}$ for all $n\geq1$.\\
\end{theorem}

\begin{theorem}[\textbf{Convergence of general ERWG in globally super-diffusive regime}]
Suppose $B$ is any interacting matrix (not necessarily diagonalisable) and $\frac{1}{2}<\eta<1$. Then there exists complex random variables $\xi_1,...,\xi_s$ such that
\begin{equation}
\label{t131}
\frac{S_n}{n^{\eta}(\log n)^{\nu-1}}-\sum\limits_{t:\Re(\lambda_t)=1-\eta,\nu_t=\nu}e^{i\Im(\lambda_t)\log n}\xi_t\mathbf{e}_t\T^{-1}\xrightarrow{a.s.}0
\end{equation}
where $i=\sqrt{-1}$ and $\mathbf{e}_t$ is the $k$-dimensional row vector whose $\bigg(\sum\limits_{j=1}^t\nu_j\bigg)$'th entry is 1 and all other entries are $0$.
\end{theorem}

\section{A special case: two incompetent elephants with same memory}

In this section, we examine a special case of the ERWG model to illustrate the significance of the general results in Section 3. Consider two elephants, labeled 1 and 2, who are walking. Lacking confidence in their own history, each relies on the other to decide new steps. Assume they both posses the same memory capacitance \( p \in [0,1] \). Specifically, at time \( n+1 \), elephant 1 looks at the past \( n \) steps of elephant 2, randomly selects one, and then takes that step with probability \( p \), or the opposite step with probability \( 1-p \). Similarly, elephant 2 follows the same procedure using the past \( n \) steps of elephant 1. This is done by both of them simultaneously and independently given their joint history up to time \( n \).\\

More precisely, let $\G$ be the directed graph with vertex set $V=\{1,2\}$ and edge set $\E=\{(1,2),(2,1)\}$. Let $q_1,q_2,p\in[0,1]$ and $S_n:=(\s{n}{1},\s{n}{2})\sim\text{ERWG}(p,q_1,q_2,\G)$, where $\s{n}{j}$ denotes the position of elephant $j$ at time $n$. The memory matrix $B$ of $S_n$ is given by
\[
B=\begin{bmatrix}
0 & 2p-1 \\
2p-1 & 0
\end{bmatrix}
\]
The eigenvalues of $B$ are $\lambda_1=2p-1$ and $\lambda_2=1-2p$. The global tuning parameter $\eta=\max\{2p-1,1-2p\}=|2p-1|$.\\

Since $B$ is symmetric, we need to take an orthogonal $T$ for its spectral decomposition to apply the results of section 2. In this case the unique orthogonal matrix $T$ satisfying $T^{-1}BT=\Lambda=\text{diag}(\lambda_1,\lambda_2)$ is given by
 \[
T=\frac{1}{\sqrt{2}}\begin{bmatrix}
1 & 1 \\
1 & -1
\end{bmatrix}
\]

The projected martingales corresponding to eigenvalues $\lambda_1=2p-1$ and $\lambda_2=1-2p$ are respectively $(\sqrt{2}d_n^{(1)})^{-1}(\s{n}{1}+\s{n}{2})$ and $(\sqrt{2}d_n^{(2)})^{-1}(\s{n}{1}-\s{n}{2})$ for all $n\geq3$.\\ 

Where $d_n^{(1)}=\frac{\Gamma(n+2p-1)}{\Gamma(2p+1)\Gamma(n)}$ and $d_n^{(2)}=\frac{\Gamma(n-2p+1)}{\Gamma(3-2p)\Gamma(n)}$ for $n\geq3$.\\
Further the projected walks corresponding to $\lambda_1$ and $\lambda_2$ are respectively $\st{n}{1}=\frac{1}{\sqrt{2}}(\s{n}{1}+\s{n}{2})$ and $\st{n}{2}=\frac{1}{\sqrt{2}}(\s{n}{1}-\s{n}{2})$.\\

Taking expectation on both sides of \eqref{1} and solving the recursion, it is not difficult to see that
\begin{equation}
\label{3.1}
\mathbb{E}[\s{n}{1}+\s{n}{2}]=2(2p-1)(q_1+q_2-1)d_n^{(1)}
\end{equation}
and 
\begin{equation}
\label{3.2}
\mathbb{E}[\s{n}{1}-\s{n}{2}]=2(1-2p)(q_1-q_2)d_n^{(2)}\\
\end{equation}

The second order moments $\mathbb{E}[(\s{n}{1})^2],~\mathbb{E}[(\s{n}{2})^2]$ and $\mathbb{E}[\s{n}{1}\s{n}{2}]$ are given by the recursive relation:
\begin{equation}
\label{3.3}
\mathbb{E}[S_{n+1}'S_{n+1}]=\left(I+\frac{B}{n}\right)\mathbb{E}[S_n'S_n]\left(I+\frac{B}{n}\right)+I-\frac{(2p-1)^2}{n^2}\text{diag}\big((\s{n}{2})^2,(\s{n}{1})^2\big)
\end{equation}
which has no close form solution. However the results in section 2 enables us to find the asymptotic order of them. We shall discuss them in detail in this section.\\

Based on the memory parameter $p$, the asymptotic behaviour of the elephants can be categorised into three regimes: 
\subsection{Diffusive regime:} We say that the two elephants are in diffusive regime if the memory capacitance $p\in(\frac{1}{4},\frac{3}{4})$, i.e., in some sense they both have moderate memory. In this regime the elephants show standard diffusive behaviour like i.i.d. random variables. For example with a standard $\sqrt{2n\log\log n}$ scaling, the joint walk exhibits a compact law of iterated logarithm:

\begin{corollary}
With probability 1, the sequence $\frac{S_n}{\sqrt{2n\log\log n}}$ is relatively compact and its set of limit points is the 2-dimensional rotated ellipsoid $E:=\{(x,y)~|~x^2-4(2p-1)xy+y^2\leq1\}$.
\end{corollary}  

Moreover, the joint walk shows asymptotic normality with standard $\sqrt{n}$ scaling:

\begin{corollary}
The following CLT holds:
\begin{equation}
\label{cor61}
\frac{S_n}{\sqrt{n}}\xrightarrow{d}\mathcal{N}_2\left((0,0),\frac{1}{1-4(2p-1)^2}\begin{bmatrix}
1 & 4p-2 \\
4p-2 & 1
\end{bmatrix}\right)\\\\
\end{equation}
\end{corollary}

Asymptotic order of the all moments are given by the following corollary:

\begin{corollary}
$\mathbb{E}\big[|\s{n}{1}\pm\s{n}{2}|^m\big]=O(n^{\frac{m}{2}})$ for all $m>0$.\\
\end{corollary}

Next corollary shows that in the long run, the sum and difference walks of the elephants behave like two independent Brownian motions with different scaling:\\

\begin{corollary} There exists a probability space where the entire sequence $S_n$ is redefined without changing its distribution along with two independent Brownian motions $\{W^{(1)}(s)\}_{s\geq0}$ and $\{W^{(2)}(s)\}_{s\geq0}$ such that for some $\delta>0$

\begin{equation}
\label{cor81}
\big(\s{n}{1}+\s{n}{2},\s{n}{1}-\s{n}{2}\big)=\sqrt{2}\bigg(n^{2p-1}W^{(1)}\big(t_n(2p-1)\big),n^{1-2p}W^{(2)}\big(t_n(1-2p)\big)\bigg)+O(n^{\frac{1}{2}-\delta})\hspace{0.25cm}a.s.
\end{equation}
where $t_n(\lambda):=\sum\limits_{j=1}^nj^{-2\lambda}$.\\
\end{corollary}

\subsection{Critical regime:} The elephants are said to be in the critical regime when the memory parameter \( p \in \left\{\frac{1}{4}, \frac{3}{4}\right\} \). These values mark the boundary between diffusive and super-diffusive behaviour. It turns out that In this regime, the elephants still display diffusive-like behaviour, but with slower rate. 

For example the sum and difference walks individually display compact law of iterated logarithm with two different scaling:

\begin{corollary}
\begin{itemize}
\item[(a)] If $p=\frac{3}{4}$, then with probability 1, both the sequences $\frac{\s{n}{1}+\s{n}{2}}{\sqrt{2n\log n\log\log\log n}}$ and $\frac{\s{n}{1}-\s{n}{2}}{\sqrt{2n\log\log n}}$ are relatively compact with set of limit points $[-\sqrt{2},\sqrt{2}]$ and $[-1,1]$ respectively. 

In particular with probability 1, the set of limit points of the sequence $\frac{S_n}{\sqrt{2n\log n\log\log\log n}}$ is $\big\{(x,x)~|~|x|\leq\frac{1}{\sqrt{2}}\big\}$.

\item[(b)] If $p=\frac{1}{4}$, then with probability 1, both the sequences $\frac{\s{n}{1}+\s{n}{2}}{\sqrt{2n\log\log n}}$ and $\frac{\s{n}{1}-\s{n}{2}}{\sqrt{2n\log n\log\log\log n}}$ are relatively compact with set of limit points $[-1,1]$ and $[-\sqrt{2},\sqrt{2}]$ respectively. 

In particular with probability 1, the set of limit points of the sequence $\frac{S_n}{\sqrt{2n\log n\log\log\log n}}$ is $\big\{(x,-x)~|~|x|\leq\frac{1}{\sqrt{2}}\big\}$.\\

\end{itemize}
\end{corollary}

Similarly with two different scaling we have the following joint CLT for the sum and difference walks:\\

\begin{corollary}
\begin{itemize}
\item[(a)] If $p=\frac{3}{4}$, then 

\begin{equation}
\label{cor101}
\left(\frac{\s{n}{1}+\s{n}{2}}{\sqrt{n\log n}},\frac{\s{n}{1}-\s{n}{2}}{\sqrt{n}}\right)\xrightarrow{d}\mathcal{N}_2\left((0,0),\begin{bmatrix}
2 & 0 \\
0 & 1
\end{bmatrix}\right)
\end{equation}
in particular the joint walk $S_n$ shows the following weak convergence
\begin{equation}
\label{cor102}
\frac{S_n}{\sqrt{n\log n}}\xrightarrow{d}\frac{1}{\sqrt{2}}(Z,Z)
\end{equation}
where $Z$ is a standard 1-dimensional normal variable.

\item[(b)] If $p=\frac{1}{4}$, then
\begin{equation}
\label{cor103}
\left(\frac{\s{n}{1}+\s{n}{2}}{\sqrt{n}},\frac{\s{n}{1}-\s{n}{2}}{\sqrt{n\log n}}\right)\xrightarrow{d}\mathcal{N}_2\left((0,0),\begin{bmatrix}
1 & 0 \\
0 & 2
\end{bmatrix}\right)
\end{equation}
in particular the joint walk $S_n$ shows the following weak convergence
\begin{equation}
\label{cor104}
\frac{S_n}{\sqrt{n\log n}}\xrightarrow{d}\frac{1}{\sqrt{2}}(Z,-Z)\\
\end{equation}
\end{itemize}
\end{corollary}

Next corollary gives the description of asymptotic order of moments:

\begin{corollary}
For any $m>0$, If $p=\frac{3}{4}$
 
then $\mathbb{E}[|\s{n}{1}+\s{n}{2}|^m]=O((n\log n)^{\frac{m}{2}})$ and $\mathbb{E}[|\s{n}{1}-\s{n}{2}|^m]=O(n^{\frac{m}{2}}),$ \\
and if $p=\frac{1}{4}$

then $\mathbb{E}[|\s{n}{1}+\s{n}{2}|^m]=O(n^{\frac{m}{2}})$ and $\mathbb{E}[|\s{n}{1}-\s{n}{2}|^m]=O((n\log n)^{\frac{m}{2}}).$\\
\end{corollary}

The following corollary shows that the sum and difference walks behaves like two independent Brownian motions with different scaling and with an error rate slightly higher than diffusive regime.  

\begin{corollary}
If $p\in\{\frac{1}{4},\frac{3}{4}\}$, then there exists a probability space where the entire sequence $S_n$ is redefined without changing its distribution along with two independent Brownian motions $\{W^{(1)}(s)\}_{s\geq0}$ and $\{W^{(2)}(s)\}_{s\geq0}$ such that for some $\delta>0$
\begin{equation}
\label{111}
\big(\s{n}{1}+\s{n}{2},\s{n}{1}-\s{n}{2}\big)=\begin{cases}
\sqrt{2}\bigg(n^{\frac{1}{2}}W^{(1)}\big(t_n(1/2)\big),n^{-\frac{1}{2}}W^{(2)}\big(t_n(-1/2)\big)\bigg)+\varepsilon_n\hspace{0.25cm}a.s., &\text{if $p=\frac{3}{4}$}\\[0.25cm]
\sqrt{2}\bigg(n^{-\frac{1}{2}}W^{(1)}\big(t_n(-1/2)\big),n^{\frac{1}{2}}W^{(2)}\big(t_n(1/2)\big)\bigg)+\tilde{\varepsilon}_n\hspace{0.25cm}a.s., &\text{if $p=\frac{1}{4}$}
\end{cases}
\end{equation}
where $\varepsilon_n=\big(O(\sqrt{n}),O(n^{\frac{1}{2}-\delta})\big)$ and $\tilde{\varepsilon}_n=\big(O(n^{\frac{1}{2}-\delta}),O(\sqrt{n})\big)$ and $t_n(.)$ is defined in Corollary 8. \\
\end{corollary}

\subsection{super-diffusive regime:} We say that the elephants are in super-diffusive regime if their memory $p\in[0,\frac{1}{4})\cup(\frac{3}{4},1]$, i.e., in some sense their memory is too high or too low. This can be divided into two scenarios:

\subsubsection{\textbf{High memory regime:}} When the memory $p\in(\frac{3}{4},1]$, we say that the  elephants have a high tendency to retain each other's past, i.e. each elephant has a high tendency to choose the past steps of other. The following result shows that they eventually go to same but random place with $d_n^{(1)}$-scaling.\\ 

\begin{corollary}
Suppose $p\in(\frac{3}{4},1]$. Then there exists a finite random variable $S_{p,q_1,q_2}$ with expectation $\mathbb{E}[S_{p,q_1,q_2}]=2p(q_1+q_2-1)$ and variance $Var(S_{p,q_1,q_2})>0$ such that 
\begin{equation}
\label{cor131}
(d_n^{(1)})^{-1}\s{n}{j}\xrightarrow{a.s.,~L^{m}}S_{p,q_1,q_2}~,~~\text{for any $m>0$ and $j=1,2$}.\\
\end{equation}
\end{corollary}

The following Theorem gives a description of the limiting random variable:

\begin{theorem}
The random variables $S_{p,q_1,q_2}$ satisfy 
\begin{equation}
\label{t141}
S_{p,q_1,q_2}\stackrel{d}{=}-S_{p,1-q_1,1-q_2}
\end{equation}
Further for $p=1$~(i.e. $\eta=2p-1=1$) we have
\begin{equation}
\label{t142}
\frac{1}{2}S_{1,q_1,q_2}\stackrel{d}{=}\bb{1}_{[\X{1}{1}=\X{1}{2}=1]}-\bb{1}_{[\X{1}{1}=\X{1}{2}=-1]}+U_1\bb{1}_{[\X{1}{1}\neq\X{1}{2}]}
\end{equation}
where the random variable $U_1$ is independent of the initial steps $\X{1}{1},\X{1}{2}$ and it's distribution does not depend on $q_1,q_2$. Further $\mathbb{E}[U_1]=0$, $\mathbb{E}[U_1^2]>0$ and $\Q(-1<U_1<1)=1$.\\
\end{theorem}

The asymptotic behaviour of the joint fluctuations with respect to the unknown random variable $S_{p,q_1,q_2}$ is given by the following result:

\begin{corollary}
Assume $p\in(\frac{3}{4},1)$. Then with probability 1, the sequence 
$$n^{2p-\frac{3}{2}}(\log \log n)^{-\frac{1}{2}}\big((d_n^{(1)})^{-1}S_n-S_{p,q_1,q_2}(1,1)\big)$$
is relatively compact with set of limit points $\{(x,x)~|~|x|\leq\sigma_1/\sqrt{2}\}$ where $\sigma_1=\frac{\Gamma(2p+1)}{\sqrt{4p-3}}$.

Further the following weak convergence holds:
\begin{equation}
 \label{cor151}
 n^{2p-\frac{3}{2}}\left((d_n^{(1)})^{-1}S_n-S_{p,q_1,q_2}(1,1)\right)~\xrightarrow{d}~\frac{\sigma_1}{\sqrt{2}}Z(1,1)
 \end{equation}
 where $Z$ is a 1-dimensional standard normal variable.\\
\end{corollary}

The difference walk shows the following  diffusive-like behaviour:

\begin{corollary}
Assume $p\in(\frac{3}{4},1)$. Then with probability 1, the sequence $$(2n\log\log n)^{-\frac{1}{2}}(\s{n}{1}-\s{n}{2})$$
is relatively compact with set of limit points $[-(2p-1/2)^{-1/2},(2p-1/2)^{-1/2}]$. Further the following CLT holds:
\begin{equation}
\label{cor161}
n^{-\frac{1}{2}}(\s{n}{1}-\s{n}{2})\xrightarrow{d}\mathcal{N}\left(0,(2p-1/2)^{-1}\right)
\end{equation}
\end{corollary}

\subsubsection{\textbf{Low memory regime:}} When the memory \( p \in [0, \frac{1}{4}) \), we say the elephants have trust issues with each other. In other words, while they follow each other, they tend to take steps in the opposite direction.

The following result shows that the elephants eventually go to an exact opposite random place with $d_n^{(2)}$-scaling.

\begin{corollary}
Assume $p\in[0,\frac{1}{4})$ then there exists a random variable $S_{p,q_1,q_2}$ with expectation $\mathbb{E}[S_{p,q_1,q_2}]=2(1-p)(q_1-q_2)$ and variance $Var(S_{p,q_1,q_2})>0$ such that 
\begin{equation}
\label{19}
(d_n^{(2)})^{-1}\s{n}{j}\xrightarrow{a.s.,~L^{p_1}}(-1)^{j+1}S_{p,q_1,q_2}~,~~\text{for any $p_1>0$ and $j=1,2.$}\\
\end{equation}
\end{corollary}

Description of the limit is given by the following Theorem similar to Theorem 14:

\begin{theorem}
The random variables $S_{p,q_1,q_2}$ satisfies \eqref{t141}. 

In particular for $p=0$~(i.e. $\eta=1-2p=1$), we have 

\begin{equation}
\label{t151}
\frac{1}{2}S_{0,q_1,q_2}\stackrel{d}{=}\bb{1}_{[\X{1}{1}=-\X{1}{2}=1]}-\bb{1}_{[\X{1}{1}=-\X{1}{2}=-1]}+U_0\bb{1}_{[\X{1}{1}=\X{1}{2}]}
\end{equation}
where $U_0$ is independent of $\X{1}{1},\X{1}{2}$ and it's distribution does not depend on $q_1,q_2$. Further $\mathbb{E}[U_0]=0$, $\mathbb{E}[U_0^2]>0$ and $\Q(-1<U_0<1)=1$.\\
\end{theorem}

The asymptotic behaviour of the joint fluctuations is given by the following result:

\begin{corollary}\singlespacing
Assume $p\in(0,\frac{1}{4})$. Then with probability 1, the sequence 
$$n^{\frac{1}{2}-2p}(\log \log n)^{-\frac{1}{2}}\big((d_n^{(2)})^{-1}S_n-S_{p,q_1,q_2}(1,-1)\big)$$
is relatively compact with set of limit points $\{(x,-x)~|~|x|\leq\sigma_2/\sqrt{2}\}$ where $\sigma_2=\frac{\Gamma(3-2p)}{\sqrt{1-4p}}$.

Further the following weak convergence holds:
\begin{equation}
 \label{cor171}
 n^{\frac{1}{2}-2p}\left((d_n^{(2)})^{-1}S_n-S_{p,q_1,q_2}(1,-1)\right)~\xrightarrow{d}~\frac{\sigma_2}{\sqrt{2}}Z(1,1)
 \end{equation}
 where $Z$ is a 1-dimensional standard normal variable.\\
\end{corollary}

The sum walk shows the following  diffusive-like behaviour:

\begin{corollary}
Assume $p\in(0,\frac{1}{4})$. Then with probability 1, the sequence $$(2n\log\log n)^{-\frac{1}{2}}(\s{n}{1}+\s{n}{2})$$
is relatively compact with set of limit points $[-(1-2p)^{-1/2},(1-2p)^{-1/2}]$. Further the following CLT holds:
\begin{equation}
\label{cor181}
n^{-\frac{1}{2}}(\s{n}{1}+\s{n}{2})\xrightarrow{d}\mathcal{N}\left(0,(1-2p)^{-1}\right)
\end{equation}
\end{corollary}
\newpage
\section{proof of the results in section 3}

We shall first prove Theorem 1. For that we will require the following Lemma:

\begin{lemma}
Suppose $\lambda\in\mathbb{C}$ such that $|\lambda|\leq1$. Let $d_{n+1}=\frac{\Gamma(\lambda+n+1)}{\Gamma(\lambda+2)\Gamma(n+1)}$ for $n\geq2$. Let $\delta_n$ be a complex-valued martingale difference sequence such that there exists a non-random $0<c<\infty$ such that $|\delta_n|<c$ a.s. for all $n$. Define 
$$L_n=\sum\limits_{l=3}^n(d_l)^{-1}\delta_l,~~\forall n\geq3.$$
Then we have the following:
\begin{equation}\label{4}
L_n=\begin{cases}
o\big(n^{\frac{1}{2}-\Re(\lambda)}(\log n)^{\frac{1}{2}+\epsilon}\big) &\text{if}~\Re(\lambda)<\frac{1}{2}\\[0.25cm]
o\big((\log n)^{\frac{1}{2}}(\log \log n)^{\frac{1}{2}+\epsilon}\big) &\text{if}~\Re(\lambda)=\frac{1}{2}
\end{cases}~~~a.s.~\text{,for every}~\epsilon>0
\end{equation} 
and \begin{equation}\label{4a}
\mathbb{E}\left[\max\limits_{l\leq n}|L_l|^m\right]=\begin{cases}
O\left(n^{m(\frac{1}{2}-\Re(\lambda))}\right) &\text{if}~\Re(\lambda)<\frac{1}{2}\\[0.25cm]
O\left((\log n)^{\frac{m}{2}}\right) &\text{if}~\Re(\lambda)=\frac{1}{2}
\end{cases}~\text{,for any $m>0$}
\end{equation} 
Finally if $\Re(\lambda)>\frac{1}{2}$ then there exists a complex random variable $L_{\infty}$ such that
\begin{equation}\label{5}
L_n\xrightarrow{a.s.,~L^m}L_{\infty}~,~\text{for any $m>0$}
\end{equation}
In particular $\mathbb{E}[|L_{\infty}|^2]>0$\hspace{0.2cm}if\hspace{0.2cm} $\liminf\limits_{n\rightarrow\infty}\mathbb{E}[|\delta_n|^2]>0$
\end{lemma}

\begin{proof} For the almost sure asymptotics of $L_n$, we adopt the standard martingale techniques used in \cite{Bercu} with slight modification to get stronger rates and for the asymptotic order of the moments, we apply the moment inequalities discussed in \cite{martingalelimit}.

Clearly $L_n$ is a mean $0$ and square integrable complex martingale. Let the quadratic variation $\langle L\rangle_n:=\sum\limits_{l=4}^n\mathbb{E}\big[|L_l-L_{l-1}|^2~\big|~L_1,...,L_{l-1}\big]$ and define $l_n=\sum\limits_{l=3}^n|d_l|^{-2}$. Then we must have 
$$\langle L\rangle_n\leq c^2l_n$$ 
First notice that 
\begin{equation}\label{gam}
d_n\sim\frac{n^{\lambda}}{\Gamma(\lambda+2)}
\end{equation}
For $\lambda\neq-1$ this comes directly from the property of $\Gamma$ function and for $\lambda=-1$ it holds trivially. We use this to obtain the asymptotic rate of $l_n$: 
\begin{equation}\label{6}
l_n\sim\begin{cases}
\frac{|\Gamma(\lambda+2)|^2}{1-2\Re(\lambda)}n^{1-2\Re(\lambda)} &\text{if}~\Re(\lambda)<\frac{1}{2}\\[0.25cm]
|\Gamma(\lambda+2)|^2\log n &\text{if}~\Re({\lambda})=\frac{1}{2}.
\end{cases}
\end{equation}
Further if $\Re(\lambda)>\frac{1}{2}$ then $l_n\rightarrow l$ for some $l\in(0,\infty)$. 

Now note that $\Re(L_n)$ and $\Im(L_n)$ are real-valued square integrable martingales satisfying $\langle \Re(L_n)\rangle\leq\langle L\rangle_n$ and $\langle \Im(L_n)\rangle\leq\langle L\rangle_n$. 

Using Theorem 1.3.15 of \cite{duflo2013random} we get
\begin{equation}
\Re(L_n)=o\big(l_n^{\frac{1}{2}}(\log l_n)^{\frac{1}{2}+\epsilon}\big)~~a.s.~~~\text{for all}~\epsilon>0\label{7}
\end{equation}
further if $\Re(\lambda)>\frac{1}{2}$ then $\Re(L_n)$ converges almost surely. Note that same thing holds for $\Im(L_n)$ too. 

Now \eqref{4} follows immediately using \eqref{6} in \eqref{7} and the almost sure convergence in \eqref{5} follows. 

Now we shall prove  \eqref{4a} and the $L^m$-convergence in \eqref{5}. 

Let $m>0$. Applying Theorem 2.11 of \cite{martingalelimit} for the martingale $\{\Re(L_n)\}_{n\geq3}$ we get
\begin{equation}\label{8}
\mathbb{E}\left[\max\limits_{l\leq n}|\Re(L_l)|^m\right]=O\left(\mathbb{E}\left[(\langle L\rangle_n)^{\frac{m}{2}}\right]+c^m\max\limits_{l\leq n}|d_l|^{-m}\right)
\end{equation} 
notice that $\max\limits_{l\leq n}|d_l|^{-1}=O\big(l_n^{\frac{1}{2}}\big)$ for any $\lambda$, hence the right hand side of \eqref{8} is $O\big(l_n^{\frac{m}{2}}\big)$. Note that same thing holds for $\mathbb{E}\left[\max\limits_{l\leq n}|\Im(L_l)|^m\right]$ also. It is worthwhile to notice that \begin{equation}\label{9}
\mathbb{E}\left[\max\limits_{l\leq n}|L_l|^m\right]\leq C_m\mathbb{E}\left[\max\limits_{l\leq n}|\Re(L_l)|^m\right]+C_m\mathbb{E}\left[\max\limits_{l\leq n}|\Im(L_l)|^m\right]
\end{equation}
where $C_m=1$ or $2^{\frac{m}{2}}$ according as $m\leq2$ or $m>2$ respectively. 

Now if $\Re(\lambda)\leq\frac{1}{2}$ then using \eqref{6} in \eqref{8}, we immediately get \eqref{4a}. Further if $\Re(\lambda)>\frac{1}{2}$ then from \eqref{8}, \eqref{9} and the fact that $l_n$ converges, we have 
$$\sup\limits_{n\geq3}\mathbb{E}\left[\max\limits_{l\leq n}|L_l|^m\right]<\infty$$
and hence by monotone convergence theorem $$\mathbb{E}\left[\sup\limits_{n\geq3}|L_n|^m\right]=\sup\limits_{n\geq3}\mathbb{E}\left[\max\limits_{l\leq n}|L_l|^m\right]<\infty$$ 

Now since $L_n\xrightarrow{a.s.}L_{\infty}$, by Fatou's lemma we get 
$$\mathbb{E}[|L_{\infty}|^m]\leq\liminf\limits_{n\rightarrow\infty}\mathbb{E}[|L_n|^m]\leq\sup\limits_{n\geq3}\mathbb{E}\left[\max\limits_{l\leq n}|L_l|^m\right]<\infty$$
Notice that 
$$|L_n-L_{\infty}|^m\leq C_{2m}|L_n|^p+C_{2m}|L_{\infty}|^m\leq C_{2m}\sup\limits_{n\geq 3}|L_n|^m+C_{2m}|L_{\infty}|^m$$
Hence by dominated convergence theorem $\mathbb{E}|L_n-L_{\infty}|^m\rightarrow0$.

Now it only remains to show that if $\liminf\limits_{n\rightarrow\infty}\mathbb{E}[|\delta_n|^2]>0$ then $\mathbb{E}[|L_{\infty}|^2]>0$. From the $L^2$ convergence of $L_n$ to $L_{\infty}$ we get 
\begin{align*}
\mathbb{E}[|L_{\infty}|^2] &=\lim\limits_{n\rightarrow\infty}\mathbb{E}[|L_n|^2]\\
&=\sum\limits_{n=3}^{\infty}|d_n|^{-2}\mathbb{E}[|\delta_n|^{2}]
\end{align*}
Since $\liminf\limits_{n\rightarrow\infty}\mathbb{E}[|\delta_n|^2]>0$ we can get $c_0>0$ such that $\mathbb{E}[|L_{\infty}|^2]\geq c_0l>0$, which completes the proof.
\end{proof}

\subsection{Proof of Theorem 1:}\singlespacing Recall the projection martingales $(d_n^{(j)})^{-1}\st{n}{j}=L_n^{(j)}+\st{2}{j}$ associated with the diagonalisable ERWG $S_n$ (see Definition 4 and 5).

Fix $j\in\{1,...,k\}$. In the set-up of Lemma 2, choose $\lambda=\lambda_j$, $\delta_n=\Delta \hat{M}_n^{(j)}$ and $d_n=d_n^{(j)}$. Note that $\delta_n=\Delta M_nTe_j'$, hence by definition of $\Delta M_n$, we can get deterministic $c\in(0,\infty)$ such that $|\delta_n|<c~a.s.$ for all $n$.

Now we apply Lemma 2 on the martingale $L_n=L_n^{(j)}$. Since $\st{n}{j}=d_n[L_n+\st{2}{j}]$ and $d_n\sim\frac{n^{\lambda_j}}{\Gamma(\lambda_j+2)}$, we immediately get \eqref{t1-1} from \eqref{4}. 

Similarly the asymptotic orders of the moments given in \eqref{t12} follows from \eqref{4a}.\\
In case $\Re(\lambda)>\frac{1}{2}$, the existence of the complex random variable $\st{\infty}{j}:=L_{\infty}+\st{2}{j}$ satisfying the almost sure and $L^m$-convergence \eqref{t11} is guaranteed by \eqref{5}.\\ 
Further $\mathbb{E}[\st{\infty}{j}]=\lim\limits_{n\rightarrow\infty}\mathbb{E}[L_n+\st{2}{j}]=\mathbb{E}[\st{2}{j}]$. Now it only remains to show that if $\eta<1$ then
$$\mathbb{E}\big[\big|\st{\infty}{j}-\mathbb{E}[\st{\infty}{j}]\big|^2\big]>0$$
Notice that $\mathbb{E}[\bar{L}_{\infty}\st{2}{j}]=\lim\limits_{n\rightarrow\infty}\mathbb{E}[\bar{L}_{n+1}\st{2}{j}]=\lim\limits_{n\rightarrow\infty}\mathbb{E}[\mathbb{E}[\bar{L}_{n+1}~|~\F{n}]\st{2}{j}]=0$ due to the $L^2$-convergence of $\bar{L}_n$ to $\bar{L}_{\infty}$. Same is true if we interchange $\bar{L}_n$ by $L_n$ and $\st{2}{j}$ by its complex conjugate. 

Hence $\mathbb{E}[|\st{\infty}{j}|^2]=\mathbb{E}[|L_{\infty}+\st{2}{j}|^2]=\mathbb{E}[|L_{\infty}|^2]+\mathbb{E}[|\st{2}{j}|^2]$. In other words 
$$\mathbb{E}\big[\big|\st{\infty}{j}-\mathbb{E}[\st{\infty}{j}]\big|^2\big]=\mathbb{E}[|L_{\infty}|^2]$$ 
Now, by Lemma 2, it is enough to show $\liminf\limits_{n\rightarrow}\mathbb{E}[|\delta_n|^2]>0$. 

First of all notice that, conditioned on $\F{n}$, the random variables $\X{n+1}{1},...,\X{n+1}{k}$ are independent and have  Rad$\big(\frac{1}{2}+\frac{1}{2}\frac{S_n}{n}Be_1'\big)$,...,Rad$\big(\frac{1}{2}+\frac{1}{2}\frac{S_n}{n}Be_k'\big)$ distributions respectively. Hence we have
\begin{equation}\label{pt11}
\begin{split}
\mathbb{E}[\Delta M_{n+1}'\Delta M_{n+1}~|~\F{n}]
&=\text{Var}(X_{n+1}~|~\F{n})\\
&=\text{diag}\left(1-\left(\frac{S_n}{n}Be_1'\right)^2,...,1-\left(\frac{S_n}{n}Be_k'\right)^2\right)
\end{split}
\end{equation}
Note that if $\eta<1$ then $\frac{\st{n}{l}}{n}\xrightarrow{a.s.}0$ for all $l=1,...,k$ and hence $\frac{S_n}{n}\xrightarrow{a.s.}0$, implying that  \begin{equation}
\label{pt12}
\mathbb{E}[\Delta M_{n+1}'\Delta M_{n+1}~|~\F{n}]\xrightarrow{a.s.}I
\end{equation}
Observe that $\mathbb{E}[|\delta_{n+1}|^2~|~\F{n}]=\mathbb{E}[\bar{\delta}_{n+1}\delta_{n+1}~|~\F{n}]=e_j\bar{T}'\mathbb{E}[\Delta M_{n+1}'\Delta M_{n+1}~|~\F{n}]Te_j'$. Hence from \eqref{pt12} we have
\begin{equation}
\label{pt13}
\mathbb{E}[|\delta_{n+1}|^2~|~\F{n}]\xrightarrow{a.s.} e_j\bar{T}'Te_j'
\end{equation}

This immediately implies (by dominated convergence theorem) $\mathbb{E}[|\delta_{n+1}|^2]\rightarrow e_j\bar{T}'Te_j'$ which must be positive since $T$ is invertible\qed\vspace{0.5cm}

Now we shall proceed to prove Theorem 2. We will require the following Lemma:

\begin{lemma}
In the set-up of Lemma 2 assume that $\lambda\in[-1,1]$ and 
$$\mathbb{E}[\delta_n^2~|~L_1,...,L_{n-1}]\xrightarrow{a.s.}\alpha$$ for some non-random $\alpha\in(0,\infty)$. Then we have the following:\\\\
If $\lambda<\frac{1}{2}$, then 
\begin{equation}
\label{l10}
n^{\lambda-\frac{1}{2}}L_n\xrightarrow{d}\mathcal{N}\left(0,\sigma_{1,\lambda}^2\right)
\end{equation}\\
and with probability 1 the sequence $\{n^{\lambda-\frac{1}{2}}(2\log \log n)^{-\frac{1}{2}}L_n\}_{n\geq3}$ is relatively compact with set of limit points $[-\sigma_{1,\lambda},\sigma_{1,\lambda}]$, where $\sigma_{1,\lambda}^2:=\frac{\alpha\Gamma(\lambda+2)^2}{1-2\lambda}$.\\\\
If $\lambda=\frac{1}{2}$, then 
\begin{equation}
\label{l11}
(\log n)^{-\frac{1}{2}}L_n\xrightarrow{d}\mathcal{N}\left(0,9\alpha\pi/16\right)
\end{equation}\\
and with probability 1 the sequence $\{(2\log n\log \log n)^{-\frac{1}{2}}L_n\}_{n\geq3}$ is relatively compact with set of limit points $\left[-3\sqrt{\alpha\pi}/4,3\sqrt{\alpha\pi}/4\right]$.\\\\
Finally if $\lambda>\frac{1}{2}$, then the almost sure limit $L_{\infty}$ satisfies $\mathbb{E}[L_{\infty}]=0$ and $\mathbb{E}[L_{\infty}^2]>0$.\\ 
Further the followings hold:
\begin{equation}
\label{l12}
n^{\lambda-\frac{1}{2}}(L_n-L_{\infty})\xrightarrow{d}\mathcal{N}(0,\sigma_{2,\lambda}^2)
\end{equation}\\
and with probability 1 the sequence $\{n^{\lambda-\frac{1}{2}}(2\log \log n)^{-\frac{1}{2}}(L_n-L_{\infty})\}_{n\geq3}$ is relatively compact with set of limit points $[-\sigma_{2,\lambda},\sigma_{2,\lambda}]$, where $\sigma_{2,\lambda}^2:=\frac{\alpha\Gamma(\lambda+2)^2}{2\lambda-1}$\\[0.25cm]
\end{lemma}

\begin{proof}
We will use the CLT and law of iterated logarithm results for partial and tail sum of martingale differences established by Heyde in \cite{centrallimitsuppliments}, which can also be found in \cite{martingalelimit}. These results were also used in \cite{gaussianflucsuper} to establish CLT and law of iterated logarithm for fluctuations of super-diffusive elephant random walk. 

Define $s_n^2(\lambda):=\sum\limits_{l=3}^nd_l^{-2}\mathbb{E}[\delta_l^2]$ if $\lambda\leq\frac{1}{2}$ and $s_n^2(\lambda):=\sum\limits_{l=n}^{\infty}d_l^{-2}\mathbb{E}[\delta_l^2]$ if $\lambda>\frac{1}{2}$. Similarly define $V_n^2(\lambda):=\sum\limits_{l=4}^nd_l^{-2}\mathbb{E}[\delta_l^2|L_1,...,L_{l-1}]$ if $\lambda\leq\frac{1}{2}$ and $V_n^2(\lambda):=\sum\limits_{l=n}^{\infty}d_l^{-2}\mathbb{E}[\delta_l^2|L_1,...,L_{l-1}]$ if $\lambda>\frac{1}{2}$.\\
Notice that $\mathbb{E}[\delta_n^2]\xrightarrow{a.s.}\alpha$ by dominated convergence theorem and hence we have \begin{equation}\label{xxx}
s_n^2(\lambda)\sim\begin{cases}\sigma_{1,\lambda}^2n^{1-2\lambda} &\text{,if $\lambda<\frac{1}{2}$}\\[0.25cm]
\alpha\frac{9\pi}{16}\log n &\text{,if $\lambda=\frac{1}{2}$}\\[0.25cm]
\sigma_{2,\lambda}^2n^{1-2\lambda} &\text{,if $\lambda>\frac{1}{2}$}
\end{cases}
\end{equation}
Further $V_n^2(\lambda)\sim s_n^2(\lambda)~a.s.$ for any $\lambda\in[-1,1]$. Now observe that
\begin{equation}\label{pl13}
\mathbb{E}\left[s_n^{-4}(\lambda)\big(d_n^{-2}\delta_n^2-\mathbb{E}[d_n^{-2}\delta_n^2~|~L_1,...,L_{n-1}]\big)^2~\bigg|~L_1,...,L_{n-1}\right]\leq c^4d_n^{-4}s_n^{-4}(\lambda)
\end{equation} 
Since $d_n^{-4}s_n^{-4}(\lambda)=O(n^{-2})$ for any $\lambda\in[-1,1]$, it follows that the series with terms in the left hand side of the inequality \eqref{pl13} converges almost surely and hence by Theorem 2.15 of \cite{martingalelimit} we conclude that the series 
$$\sum\limits_{n=4}^{\infty}s_n^{-2}(\lambda)(d_n^{-2}\delta_n^2-\mathbb{E}[d_n^{-2}\delta_n^2|L_1,...,L_{n-1}])$$
converges almost surely for any $\lambda\in[-1,1]$. Now if $\lambda\leq\frac{1}{2}$ then 
$$s_n^{-2}(\lambda)\sum\limits_{l=4}^n(d_l^{-2}\delta_l^2-\mathbb{E}[d_l^{-2}\delta_l^2|L_1,...,L_{l-1}])\xrightarrow{a.s.}0$$
by Kronecker's lemma and if $\lambda>\frac{1}{2}$ then 
$$s_n^{-2}(\lambda)\sum\limits_{l=n}^{\infty}(d_l^{-2}\delta_l^2-\mathbb{E}[d_l^{-2}\delta_l^2|L_1,...,L_{l-1}])\xrightarrow{a.s.}0$$ by the second part of Lemma 1 of \cite{centrallimitsuppliments}. In other if $\lambda\leq\frac{1}{2}$ then we have have 
$$s_n^{-2}(\lambda)\sum\limits_{l=4}^nd_l^{-2}\delta_l^2-s_n^{-2}(\lambda)V_n^2(\lambda)\xrightarrow{a.s.}0$$ and if $\lambda>\frac{1}{2}$ then we have
$$s_n^{-2}(\lambda)\sum\limits_{l=n}^{\infty}d_l^{-2}\delta_l^2-s_n^{-2}(\lambda)V_n^2(\lambda)\xrightarrow{a.s.}0$$ Hence defining $W_n^2(\lambda):=\sum\limits_{l=4}^nd_l^{-2}\delta_l^2$ if $\lambda\leq\frac{1}{2}$ and $W_n^2(\lambda):=\sum\limits_{l=n}^{\infty}d_l^{-2}\delta_l^2$ if $\lambda>\frac{1}{2}$, we must have
$$s_n(\lambda)^{-2}W_n^2(\lambda)\xrightarrow{a.s.}1$$
 Now note that $\max\limits_{l\leq n}d_l^{-2}\delta_l^2=O(\max(d_3^{-2},n^{-2\lambda}))~a.s.$ implying $\mathbb{E}\left[s_n(\lambda)^{-2}\max\limits_{l\leq n}d_l^{-2}\delta_l^2\right]\rightarrow0$ for $\lambda\leq\frac{1}{2}$. Similarly if $\lambda>\frac{1}{2}$ then notice that $\max\limits_{l\geq n}d_l^{-2}\delta_l^2=O(n^{-2\lambda})~a.s.$ which implies $\mathbb{E}\left[s_n(\lambda)^{-2}\max\limits_{l\geq n}d_l^{-2}\delta_l^2\right]\rightarrow0.$ We will now show that for any $\lambda\in[-1,1]$
\begin{equation}
\label{pl14}
\sum\limits_{n=3}^{\infty}s_n(\lambda)^{-1}\mathbb{E}[d_n^{-1}|\delta_n|\bb{1}\{d_n^{-1}|\delta_n|>\epsilon s_n(\lambda)\}]<\infty,~~~~\text{for all $\epsilon>0$}
\end{equation}
and 
\begin{equation}
\label{pl15}
\sum\limits_{n=3}^{\infty}s_n(\lambda)^{-4}\mathbb{E}[d_n^{-4}|\delta_n|^4\bb{1}\{d_n^{-1}|\delta_n|\leq\epsilon s_n(\lambda)\}]<\infty,~~~~\text{for all $\epsilon>0$}
\end{equation}
Notice that the summand in \eqref{pl14} can not be more than $cd_n^{-1}s_n(\lambda)^{-1}\Q\big(|\delta_n|>\epsilon d_ns_n(\lambda)\big)$ which is eventually $0$ for large $n$ since $d_ns_n(\lambda)\rightarrow\infty$ and $|\delta_n|<c~a.s.$ for all $n$. Hence \eqref{pl14} holds. \eqref{pl15} follows by noticing that the summand is not more than $c^4d_n^{-4}s_n(\lambda)^{-4}=O(n^{-2})$. 

Now with all these we observe that if $\lambda\leq\frac{1}{2}$ then the martingale $L_n=\sum\limits_{l=3}^nd_l^{-1}\delta_l$ satisfies all the conditions of part $(a)$ of Theorem 1 of \cite{centrallimitsuppliments} and if $\lambda>\frac{1}{2}$ then all conditions of part $(b)$ are fulfilled by the tail sum of martingale differences $L_{\infty}-L_{n-1}=\sum\limits_{l=n}^{\infty}d_l^{-1}\delta_l$. Hence if $\lambda\leq\frac{1}{2}$ then $s_n(\lambda)^{-1}L_n\xrightarrow{d}\mathcal{N}(0,1)$ which proves \eqref{l10} and \eqref{l11}. Further the sequence $W_n(\lambda)^{-1}(2\log \log W_n(\lambda))^{-\frac{1}{2}}L_n$ has for its set of almost sure limit points the closed interval $[-1,1]$. 

Now since $W_n(\lambda)^2\sim s_n(\lambda)^2~a.s.$, it immediately follows that with probability $1$ the sequence $n^{\lambda-\frac{1}{2}}(2\log \log n)^{-\frac{1}{2}}L_n$~(respectively $(2\log n\log\log\log n)^{-\frac{1}{2}}L_n$) is relatively compact with set of limit points $[-\sigma_{1,\lambda},\sigma_{1,\lambda}]$~(respectively $\left[-3\sqrt{\alpha\pi}/4,3\sqrt{\alpha\pi}/4\right]$) if $\lambda<\frac{1}{2}$ (respectively $\lambda=\frac{1}{2}$).

Similarly if $\lambda>\frac{1}{2}$  then from part (b) of Theorem 1 of \cite{centrallimitsuppliments}, the weak convergence in \eqref{l12} and the relative compactness with compact law of iterated logarithm of the sequence $$n^{\lambda-\frac{1}{2}}(2\log\log n)^{-\frac{1}{2}}(L_n-L_{\infty})$$ stated in the last part of the Lemma, follows.

Lastly the fact that $\mathbb{E}[L_{\infty}]=0$ and $\mathbb{E}[L_{\infty}^2]>0$, follows directly from Lemma 2 since $\mathbb{E}[\delta_n^2]\rightarrow\alpha>0$.
\end{proof}

\subsection{Proof of Theorem 2:} Fix $j\in\{1,...,k\}$. In the set-up of Lemma 3, let $\lambda=\lambda_j$, $\delta_n=\hat{M}_n^{(j)}$ as in the proof of Theorem 1. Hence $L_n=L_n^{(j)}=(d_n^{(j)})^{-1}\st{n}{j}-\st{2}{j}$. 

Recall the convergence \eqref{pt13} established in the proof of Theorem 1 assuming $\eta<1$, which means $\mathbb{E}[\delta_{n}^2~|~L_1,...,L_{n-1}]\xrightarrow{a.s.}e_jT'Te_j'=:\alpha$. Clearly $\alpha\in(0,\infty)$. 

Now the Theorem follows by applying Lemma 3 on the martingale $L_n$ and noticing that $\st{\infty}{j}=L_{\infty}+\st{2}{j}$ in case $\frac{1}{2}<\lambda_j<1$.\qed\\\\

The following lemma describes how the sum of  martingale differences $\Delta M_n$ can be approximated by independent normal random vectors in globally diffusive and critical regimes for diagonalisable ERWG $S_n$. This will be used to prove Theorem 3. 

\begin{lemma}
If $\eta\leq\frac{1}{2}$ and $B$ is diagonalisable then there exists a probability space where the entire sequence $\{S_n\}_{n\geq0}$ is redefined without changing it's distribution and in the same space there exists a sequence $\{Y_n\}_{n\geq2}$ of i.i.d. $\mathcal{N}_k(0,I)$ random vectors such that for some $\delta>0$

\begin{equation}\label{11}
\sum\limits_{j=1}^n\Delta M_{j+1}=\sum\limits_{j=1}^nY_{j+1}+O\big(n^{\frac{1}{2}-\delta}\big)\hspace{0.25cm}a.s.
\end{equation}
\end{lemma}

\begin{proof}
Let $\Sigma_n:=\sum\limits_{j=1}^n\mathbb{E}[\Delta M_{j+1}'\Delta M_{j+1}~|~\F{j}]$. It is easy to see from the definition of $\Delta M_{j+1}$  that 
\begin{equation}\label{12}
\Sigma_n=\sum\limits_{j=1}^n\text{diag}\left(1-\left(\frac{S_j}{j}Be_1'\right)^2,...,1-\left(\frac{S_j}{j}Be_k'\right)^2\right)
\end{equation}
by Corollary 1 we have 
\begin{equation}\label{l31}
\frac{1}{n}\Sigma_n\xrightarrow{a.s.}I
\end{equation}
From \eqref{12} we get
\begin{equation}\label{13}
\lVert \Sigma_n-nI\rVert\leq C\sum\limits_{j=1}^n\bigg\lVert\frac{S_j}{j}\bigg\rVert
\end{equation}
for some non-random $C>0$ which depends only on $B$. Since $\eta\leq\frac{1}{2}$ and $n^{-\frac{1}{2}}(\log n)^{\frac{1}{2}+\epsilon}=o(n^{-\theta})$ for any $0<\theta<\frac{1}{2}$, from \eqref{10} we get
\begin{equation}\label{15}
\frac{S_n}{n}=o(n^{-\theta})~~a.s.~\text{for all}~0<\theta<\frac{1}{2}
\end{equation}
note that by Proposition 1 
\begin{equation*}
\begin{split}
\sum\limits_{j=1}^n\bigg\lVert\frac{S_j}{j}\bigg\rVert &=o\bigg(\sum\limits_{j=1}^nj^{-\theta}\bigg)~~a.s.\\
&=o(n^{1-\theta})
\end{split}
\end{equation*}
hence we conclude that
\begin{equation}\label{16}
\lVert \Sigma_n-nI\rVert=o(n^{1-\theta})~~a.s.~\text{for any}~0<\theta<\frac{1}{2}
\end{equation} 
now since the martingale differences $\Delta M_n$ are uniformly bounded by some deterministic constant, following holds trivially:
\begin{equation}\label{17}
\sum\limits_{n=2}^{\infty}\frac{1}{n^{1-\epsilon}}\mathbb{E}\big[\lVert\Delta M_n\rVert^2\bb{1}\big\{\lVert \Delta M_n\rVert>n^{1-\epsilon}\big\}~\big|~\F{n-1}\big]<\infty~~a.s.~\text{for any}~0<\epsilon<1
\end{equation}
hence all conditions of Theorem 1.3 of \cite{zhang2004strong} are satisfied with covariance matrix $\mathbf{T}=I$, so there exists i.i.d. $\mathcal{N}(0,I)$ random vectors $Y_2,...,Y_n,...$ (possibly in some different probability space where the entire sequence $\{S_n\}_{n\geq0}$ is redefined without changing it's distribution) such that \eqref{11} holds for some $\delta>0$.
\end{proof}

\subsection{Proof of Theorem 3:}
Let $Z_n=\frac{S_n}{n}$, then the following recursion can be obtained easily from \eqref{1}:

\begin{equation}\label{20}
Z_{n+1}=Z_n\bigg(\frac{n}{n+1}I+\frac{B}{n+1}\bigg)+\frac{\Delta M_{n+1}}{n+1},\hspace{0.25cm}\text{for}~n\geq1
\end{equation} 
continuing the iteration upto time $n=1$ we get
\begin{equation}\label{21}
Z_n=X_1\C{1}{n}+\sum\limits_{j=2}^n\Delta M_j\frac{\C{j}{n}}{j},\hspace{0.25cm}\text{for}~n\geq2
\end{equation}
where 
\begin{equation}\label{C}
\C{j}{n}=\begin{cases}
\bigg(\frac{j}{j+1}I+\frac{B}{j+1}\bigg)\bigg(\frac{j+1}{j+2}I+\frac{B}{j+2}\bigg)...\bigg(\frac{n-1}{n}I+\frac{B}{n}\bigg) &,\text{if}~1\leq j<n\\[0.25cm]
I &,\text{if}~j=n
\end{cases}
\end{equation}
Since $B$ is diagonalisable, we know by Lemma 4 that there exist i.i.d. $\mathcal{N}(0,I)$ random vectors $Y_1,...,Y_n,...$ (possibly in some different probability space where the sequence $\{S_n\}_{n\geq0}$ is redefined without changing it's distribution) such that 
\begin{equation*}
\sum\limits_{j=2}^n\Delta M_j=\sum\limits_{j=2}^nY_j+R_n~~a.s.
\end{equation*}
where \begin{equation}\label{r}R_n=O\big(n^{\frac{1}{2}-\delta'}\big)\hspace{0.25cm}a.s.~~~\text{for some}~\delta'>0
\end{equation}
 Now note that for $n\geq3$
\begin{align*}
Z_n-X_1\C{1}{n}&=\sum\limits_{j=2}^n\Delta M_j\frac{\C{j}{n}}{j}\\
&=\sum\limits_{j=2}^{n-1}\Delta M_j\bigg(\sum\limits_{l=j}^{n-1}\bigg(\frac{\C{l}{n}}{l}-\frac{\C{l+1}{n}}{l+1}\bigg)\bigg)+\frac{1}{n}\sum\limits_{j=2}^n\Delta M_j\\
&=\sum\limits_{l=2}^{n-1}\sum\limits_{j=2}^l\Delta M_j\bigg(\frac{\C{l}{n}}{l}-\frac{\C{l+1}{n}}{l+1}\bigg)+\frac{1}{n}\sum\limits_{j=2}^n\Delta M_j\\
&=\sum\limits_{l=2}^{n-1}\sum\limits_{j=2}^lY_j\bigg(\frac{\C{l}{n}}{l}-\frac{\C{l+1}{n}}{l+1}\bigg)+\frac{1}{n}\sum\limits_{j=2}^nY_j+U_n\\
&=\sum\limits_{j=2}^n Y_j\frac{\C{j}{n}}{j}+U_n
\end{align*}
where 
\begin{equation}\label{22}
\begin{split}
U_n&=\sum\limits_{l=2}^{n-1}R_l\bigg(\frac{\C{l}{n}}{l}-\frac{\C{l+1}{n}}{l+1}\bigg)+\frac{1}{n}R_n\\
&=\sum\limits_{l=2}^{n-1}R_l\bigg(\frac{1}{l}\bigg(\frac{l}{l+1}I+\frac{B}{l+1}\bigg)-\frac{1}{l+1}I\bigg)\C{l+1}{n}+\frac{1}{n}R_n\\
&=\sum\limits_{l=2}^{n-1}R_l\bigg(\frac{B}{l(l+1)}\bigg)\C{l+1}{n}+\frac{1}{n}R_n
\end{split}
\end{equation}
hence taking norm we get
\begin{equation}\label{23}
\begin{split}
\lVert U_n\rVert &\leq\lVert B\rVert\sum\limits_{l=2}^{n-1}\frac{1}{l(l+1)}\lVert R_l\rVert\lVert\C{l+1}{n}\rVert+\frac{1}{n}\lVert R_n\rVert\\
&=O\bigg(\sum\limits_{l=2}^{n-1}\frac{l^{\frac{1}{2}-\delta'}}{l^2}\bigg(\frac{l+1}{n}\bigg)^{1-\eta}\bigg)+O\bigg(\frac{1}{n^{\frac{1}{2}+\delta'}}\bigg)\\
&=O\bigg(\frac{1}{n^{1-\eta}}\sum\limits_{l=2}^{n-1}\frac{1}{l^{\frac{1}{2}+\delta'+\eta}}\bigg)+O\bigg(\frac{1}{n^{\frac{1}{2}+\delta'}}\bigg)
\end{split}
\end{equation}
second line in \eqref{23} follows from \eqref{r} and (3.1) of Proposition 3. Note that as $n\rightarrow\infty$
\begin{equation}\label{24}
\frac{1}{n^{1-\eta}}\sum\limits_{l=2}^{n-1}\frac{1}{l^{\frac{1}{2}+\delta'+\eta}}\sim\begin{cases}
\bigg(\frac{1}{\frac{1}{2}-\delta'-\eta}\bigg)\frac{1}{n^{\frac{1}{2}+\delta'}}, &\text{if}~\delta'+\eta<\frac{1}{2}\\[0.25cm]
\frac{\log n}{n^{1-\eta}}, &\text{if}~\delta'+\eta=\frac{1}{2}\\[0.25cm]
\frac{O(1)}{n^{1-\eta}}, &\text{if}~\delta'+\eta>\frac{1}{2}
\end{cases}
\end{equation}
hence if $\eta<\frac{1}{2}$ then we can always get a $0<\delta<\min\{\delta',\frac{1}{2}-\eta,\frac{1}{2}\}$ such that 
$$\lVert U_n\rVert=O(n^{-\frac{1}{2}-\delta})~a.s.$$
Further if $\eta=\frac{1}{2}$ then 
$$\lVert U_n\rVert=O(n^{-\frac{1}{2}})~a.s.$$
Now note that by (3.1) of Proposition 3 we have 
$(X_1-Y_1)\C{1}{n}=O(n^{\eta-1})~a.s.$, hence 

\begin{equation}\label{bp}
\bigg\lVert Z_n-\sum\limits_{j=1}^n Y_j\frac{\C{j}{n}}{j}\bigg\rVert=\begin{cases}O(n^{-\frac{1}{2}-\delta})~a.s. &\text{,if}~\eta<\frac{1}{2}\\[0.25cm]
O(n^{-\frac{1}{2}})~a.s. &\text{,if}~\eta=\frac{1}{2}
\end{cases}
\end{equation}
Next we will show that
\begin{equation}\label{ap}
\bigg\lVert\sum\limits_{j=1}^n\frac{Y_j}{j}\bigg(\C{j}{n}-\bigg(\frac{j}{n}\bigg)^{I-B}\bigg)\bigg\rVert=\begin{cases}O(n^{-\frac{1}{2}-\delta})~a.s. &\text{,if}~\eta<\frac{1}{2}\\[0.25cm]
O(n^{-\frac{1}{2}})~a.s. &\text{,if}~\eta=\frac{1}{2}
\end{cases}
\end{equation}
Notice that the left hand side of \eqref{ap} is not more than $U_n':=\sum\limits_{j=1}^n\big\lVert\frac{Y_j}{j}\big\rVert\big\lVert\C{j}{n}-\big(\frac{j}{n}\big)^{I-B}\big\rVert$. Using (3.2) of Proposition 3 we get
\begin{equation}\label{28}
U_n'\leq\frac{1}{n^{1-\eta}}\sum\limits_{j=1}^n\frac{\lVert Y_j\rVert}{j^{2+\eta}}
\end{equation}
If $\eta>-1$ then using Fubini's theorem
$$\mathbb{E}\bigg[\sum\limits_{j=1}^{\infty}\frac{\lVert Y_j\rVert}{j^{2+\eta}}\bigg]=\sum\limits_{j=1}^{\infty}\frac{\mathbb{E}[\lVert Y_j\rVert]}{j^{2+\eta}}\leq k^{\frac{1}{2}}\sum\limits_{j=1}^{\infty}\frac{1}{j^{2+\eta}}<\infty$$
that is $\sum\limits_{j=1}^{\infty}\frac{\lVert Y_j\rVert}{j^{2+\eta}}$ is a finite random variable, which means from \eqref{28} that $U_n'=O(n^{\eta-1})$.\\

Now if $\eta=-1$ then note that for any $\epsilon>0$
$$\Q\bigg(\sum\limits_{j=1}^{n}\frac{\lVert Y_j\rVert}{j}>\epsilon n^{1-\delta}\bigg)\leq\frac{1}{\epsilon^2n^{2-2\delta}}\mathbb{E}\bigg[\bigg(\sum\limits_{j=1}^{n}\frac{\lVert Y_j\rVert}{j}\bigg)^2\bigg]\leq\frac{k\left(\sum\limits_{j=1}^n\frac{1}{j}\right)^2}{\epsilon^2n^{2-2\delta}}$$
recall that $\delta<\frac{1}{2}$. Since $\big(\sum\limits_{j=1}^n\frac{1}{j}\big)^2=o(n^{\frac{1}{2}-\delta})$ as $n\rightarrow\infty$, we must have 
$$\frac{k\left(\sum\limits_{j=1}^n\frac{1}{j}\right)^2}{\epsilon^2n^{2-2\delta}}=o(n^{\delta-\frac{3}{2}})$$
which means that the series $\sum\limits_{n=1}^{\infty}\Q\bigg(\sum\limits_{j=1}^{n}\frac{\lVert Y_j\rVert}{j}>\epsilon n^{1-\delta}\bigg)$ is convergent, hence by Borel-Cantelli lemma we conclude that
$$\sum\limits_{j=1}^{n}\frac{\lVert Y_j\rVert}{j}=o(n^{1-\delta})\hspace{0.25cm}a.s.$$
using this in the right hand side of \eqref{28} for $\eta=-1$ we get $U_n'=o(n^{-1-\delta})~a.s.$ and this proves the identity in \eqref{ap}.\\

Now combining \eqref{bp} and \eqref{ap} we get

\begin{equation}\label{eq}
\frac{S_n}{n}=\begin{cases} 
\sum\limits_{j=1}^n\frac{Y_j}{j}\big(\frac{j}{n}\big)^{I-B}+O\big(n^{-\frac{1}{2}-\delta}\big)~~a.s. &\text{,if}~\eta<\frac{1}{2}\\[0.25cm]
\sum\limits_{j=1}^n\frac{Y_j}{j}\big(\frac{j}{n}\big)^{I-B}+O\big(n^{-\frac{1}{2}}\big)~~a.s. &\text{,if}~\eta=\frac{1}{2}
\end{cases}
\end{equation}
Note that since the matrices $-I$ and $I-B$ commute with each other, by the properties of matrix exponentials, we have $\frac{n}{j}(\frac{j}{n})^{I-B}=(\frac{j}{n})^{-I}(\frac{j}{n})^{I-B}=(\frac{j}{n})^{-B}$. This proves that \eqref{eq} is equivalent to the expressions of $S_n$ in \eqref{18}.\\\\
Now shall proceed to prove \eqref{d}. Assume that $B$ is diagonalisable in $\mathbb{R}^{k\times k}$, that is all $\lambda_1,...,\lambda_k$ are real and $T$ is a real invertible matrix.\\

Note that, for $n\geq2$, we can write
\begin{equation}\label{cp2}
\frac{S_n}{n}=\sum\limits_{j=1}^n \frac{Y_j}{j}\bigg(\frac{j}{n}\bigg)^{I-B}+U_n+U_n''+(X_1-Y_1)\C{1}{n}~a.s.
\end{equation}
where $U_n$ is defined as in \eqref{22} and 
\begin{equation*}
U_n'':=\sum\limits_{j=1}^n\frac{Y_j}{j}\bigg(\C{j}{n}-\bigg(\frac{j}{n}\bigg)^{I-B}\bigg)
\end{equation*} By (2.1) of Proposition 2, it is easy to see that the last part of \eqref{cp2} is $(O(n^{\lambda_1-1}),...,O(n^{\lambda_k-1}))T^{-1}$.\\

Further from the last line of \eqref{22}, using (2.1) of Proposition 2 again, we get

\begin{equation}\label{33}
\begin{split}U_n&=\sum\limits_{l=2}^{n-1}\frac{O(l^{\frac{1}{2}-\delta'})}{l(l+1)}BT\text{diag}\bigg(\bigg(\frac{l}{n}\bigg)^{1-\lambda_1}O(1),...,\bigg(\frac{l}{n}\bigg)^{1-\lambda_k}O(1)\bigg)T^{-1}+O\bigg(\frac{1}{n^{\frac{1}{2}+\delta'}}\bigg)\\
&=\sum\limits_{l=2}^{n-1}O\bigg(\frac{1}{l^{\frac{3}{2}+\delta'}}\bigg)T\Lambda\text{diag}\bigg(\bigg(\frac{l}{n}\bigg)^{1-\lambda_1}O(1),...,\bigg(\frac{l}{n}\bigg)^{1-\lambda_k}O(1)\bigg)T^{-1}+O\bigg(\frac{1}{n^{\frac{1}{2}+\delta'}}\bigg)\\
&=\sum\limits_{l=2}^{n-1}\bigg(\frac{1}{l^{\frac{1}{2}+\lambda_1+\delta'}}\frac{O(1)}{n^{1-\lambda_1}},...,\frac{1}{l^{\frac{1}{2}+\lambda_k+\delta'}}\frac{O(1)}{n^{1-\lambda_k}}\bigg)T^{-1}+O\bigg(\frac{1}{n^{\frac{1}{2}+\delta'}}\bigg)
\end{split}
\end{equation}
Since $\lambda_j\leq\frac{1}{2}$, using \eqref{24} we can find $\delta_j\geq0$ such that $$\sum\limits_{l=2}^{n-1}\frac{1}{l^{\frac{1}{2}+\lambda_j+\delta'}}\frac{O(1)}{n^{1-\lambda_j}}=O(n^{-\frac{1}{2}-\delta_j})$$
where $0<\delta_j<\min\{\delta',\frac{1}{2}-\lambda_j,\frac{1}{2}\}$ if $\lambda_j<\frac{1}{2}$ and $\delta_j=0$ if $\lambda_j=\frac{1}{2}$.\\

Hence we have
\begin{equation}\label{34}
U_n=(O(n^{-\frac{1}{2}-\delta_1}),...,O(n^{-\frac{1}{2}-\delta_k}))T^{-1}
\end{equation}
Next we will show that
\begin{equation}\label{as}
U_n''=(O(n^{\lambda_1-1}),...,O(n^{\lambda_k-1}))T^{-1}
\end{equation}
Let $\hat{Y}_j=Y_jT$. Note that by (2.2) of Proposition 2, we have
\begin{align*}
U_n''&=\sum\limits_{j=1}^n\hat{Y}_j\text{diag}\bigg(\bigg(\frac{j}{n}\bigg)^{1-\lambda_1}\frac{O(1)}{j^3},...,\bigg(\frac{j}{n}\bigg)^{1-\lambda_k}\frac{O(1)}{j^3}\bigg)T^{-1}\\
&=\bigg(\frac{1}{n^{1-\lambda_1}}\sum\limits_{j=1}^n\frac{\hat{Y}_j^{(1)}}{j^{2+\lambda_1}},...,\frac{1}{n^{1-\lambda_k}}\sum\limits_{j=1}^n\frac{\hat{Y}_j^{(1)}}{j^{2+\lambda_k}}\bigg)T^{-1}
\end{align*} 
Hence to prove \eqref{as}, it is enough to show that the martingale $M_n^{(j)}:=\sum\limits_{j=1}^n\frac{\hat{Y}_j^{(j)}}{j^{2+\lambda_j}}$ converges almost surely. In fact it converges both almost surely and in mean square and this fact comes immediately after applying the second part of proposition 1.3.7 of \cite{duflo2013random} to the martingale $M_n^{(j)}$.\\

Finally using all these orders in \eqref{cp2} we get
$$\frac{S_n}{n}=\sum\limits_{j=1}^n\frac{Y_j}{j}\bigg(\frac{j}{n}\bigg)^{I-B}+(O(n^{-\frac{1}{2}-\delta_1}),...,O(n^{-\frac{1}{2}-\delta_k}))T^{-1}~a.s.$$
Multiplying both sides by $n$ yields \eqref{d}.\qed\\\\

With the strong Gaussian approximation in hand, it will be easy to prove the CLT stated in Theorem 4.
\subsection{Proof of Theorem 4:} Define $a_n=n^{-\frac{1}{2}}$ if $\eta<\frac{1}{2}$ and $a_n=(n\log n)^{-\frac{1}{2}}$ if $\eta=\frac{1}{2}$. By Theorem 3 we can get $Y_1,...,Y_n,...$ i.i.d. $\mathcal{N}_k(0,I)$ vectors (possibly in some different probability space where $\{S_n\}_{n\geq0}$ is redefined without changing it's distribution) such that
\begin{equation}\label{c21}
a_nS_n=\Theta_n+\zeta_n~a.s.
\end{equation} 
where $$\Theta_n=a_n\sum\limits_{j=1}^n Y_j\bigg(\frac{j}{n}\bigg)^{-B}$$
and 
\begin{equation*}
\zeta_n=\begin{cases}
O(n^{-\delta})~a.s. &\text{,if~}\eta<\frac{1}{2}\\[0.25cm]
O((\log n)^{-\frac{1}{2}})~a.s. &\text{,if~}\eta=\frac{1}{2}
\end{cases}
\end{equation*}
Note that $\Theta_n$ has a $\mathcal{N}_k(0,\Sigma_n)$ distribution where 
\begin{equation}\label{c22}
\Sigma_n=a_n^2\sum\limits_{j=1}^n\bigg(\frac{j}{n}\bigg)^{-B'}\bigg(\frac{j}{n}\bigg)^{-B}
\end{equation}
We will show that $\Sigma_n$ converges to $\Sigma^{(1)}$ if $\eta<\frac{1}{2}$ and $\Sigma_n$ converges to $\Sigma^{(2)}$ if $\eta=\frac{1}{2}$.\\

In fact we will prove the following:
\begin{equation}\label{c23}
\Sigma_n=\begin{cases}
\Sigma^{(1)}+O(n^{-1}) &\text{,if}~\eta<\frac{1}{2}\\[0.25cm]
\Sigma^{(2)}+O((\log n)^{-1}) &\text{,if}~\eta=\frac{1}{2}
\end{cases}
\end{equation}
First notice that for $n\geq2$
\begin{equation}\label{in}
\begin{split}
 &\bigg\lVert\sum\limits_{j=1}^{n-1}\bigg(\frac{j}{n}\bigg)^{-B'}\bigg(\frac{j}{n}\bigg)^{-B}-\int_1^n\bigg(\frac{x}{n}\bigg)^{-B'}\bigg(\frac{x}{n}\bigg)^{-B}dx\bigg\rVert\\ \leq&\sum\limits_{j=1}^{n-1}\int_j^{j+1}\bigg\lVert\bigg(\frac{j}{n}\bigg)^{-B'}\bigg(\frac{j}{n}\bigg)^{-B}-\bigg(\frac{x}{n}\bigg)^{-B'}\bigg(\frac{x}{n}\bigg)^{-B}\bigg\rVert dx\\
\leq&\sum\limits_{j=1}^{n-1}\int_j^{j+1}\bigg(\bigg\lVert\bigg(\frac{j}{n}\bigg)^{-B'}\bigg\rVert\bigg\lVert\bigg(\frac{j}{n}\bigg)^{-B}-\bigg(\frac{x}{n}\bigg)^{-B}\bigg\rVert+\bigg\lVert\bigg(\frac{j}{n}\bigg)^{-B'}-\bigg(\frac{x}{n}\bigg)^{-B'}\bigg\rVert\bigg\lVert\bigg(\frac{x}{n}\bigg)^{-B}\bigg\rVert\bigg)dx\\
\leq&c''\sum\limits_{j=1}^{n-1}\int_j^{j+1}\bigg(\bigg(\frac{j}{n}\bigg)^{-\eta}\frac{1}{j}\bigg(\frac{j}{n}\bigg)^{-\eta}+\frac{1}{j}\bigg(\frac{j}{n}\bigg)^{-\eta}\bigg(\frac{j}{n}\bigg)^{-\eta}\bigg)dx\\
\leq&2c''\sum\limits_{j=1}^{n-1}\frac{1}{j}\bigg(\frac{j}{n}\bigg)^{-2\eta}
\end{split}
\end{equation}
the 4th line above follows from (3.3) of Proposition 3 and the fact that $$\sup\limits_{x\in[j,j+1]}\bigg\lVert\bigg(\frac{x}{n}\bigg)^{-B}\bigg\rVert\leq\bigg(\frac{j}{n}\bigg)^{-\eta}$$ 
it is easy to see that, as $n\rightarrow\infty$
\begin{equation*}
\sum\limits_{j=1}^{n-1}\frac{1}{j}\bigg(\frac{j}{n}\bigg)^{-2\eta}\sim\begin{cases}
\frac{1}{1-2\eta}, &\text{if}~\eta<\frac{1}{2}\\[0.25cm]
\frac{\pi^2}{6}n, &\text{if}~\eta=\frac{1}{2}.
\end{cases}
\end{equation*}
hence from \eqref{in} we get
\begin{equation*}
\Sigma_n=a_n^2\int_1^n\bigg(\frac{x}{n}\bigg)^{-B'}\bigg(\frac{x}{n}\bigg)^{-B}dx+E_n
\end{equation*}
where the asymptotic order the matrix $E_n$ is given by
\begin{equation*}
E_n=\begin{cases}
O(n^{-1}) &\text{,if~}\eta<\frac{1}{2}\\[0.25cm] O((\log n)^{-1}) &\text{,if~}\eta=\frac{1}{2}
\end{cases}
\end{equation*}
The integral can be calculated by taking the transformation $\frac{x}{n}=e^{-t}$ for $x\in[1,n]$, which yields
\begin{align*}
\int_1^n\bigg(\frac{x}{n}\bigg)^{-B'}\bigg(\frac{x}{n}\bigg)^{-B}dx=n\int_0^{\log n}e^{(B-\frac{1}{2}I)'t}e^{(B-\frac{1}{2}I)t}dt
\end{align*}
this proves \eqref{c23}.\\\\
Now from \eqref{c21}, since $\zeta_n\xrightarrow{a.s.}0$, by Slutsky's theorem, weak convergences stated in \eqref{clt1} and \eqref{clt2} follow .\\

Next suppose $B$ is symmetric and $T$ is orthogonal. First let $\eta<\frac{1}{2}$.\\[0.25cm]
Since $(B-\frac{1}{2}I)'$ and $B-\frac{1}{2}I$ commute with each other, we have
\begin{align*}
\Sigma^{(1)}&=\int_0^{\infty}e^{(2B-I)t}dt\\
&=T\text{diag}\bigg(\int_0^{\infty}e^{(2\lambda_1-1)t}dt,...,\int_0^{\infty}e^{(2\lambda_k-1)t}dt\bigg)T'\\
&=T\text{diag}\bigg(\frac{1}{1-2\lambda_1},...,\frac{1}{1-2\lambda_k}\bigg)T'\\
&=T(I-2\Lambda)^{-1}T'\\
&=(I-2B)^{-1}
\end{align*}
Similarly if $\eta=\frac{1}{2}$ then 
\begin{align*}
\Sigma^{(2)}&=T\text{diag}\bigg(\lim\limits_{n\rightarrow\infty}\frac{1}{\log n}\int_0^{\log n}e^{(2\lambda_1-1)t}dt,...,\lim\limits_{n\rightarrow\infty}\frac{1}{\log n}\int_0^{\log n}e^{(2\lambda_k-1)t}dt\bigg)T'\\
&=T\text{diag}\bigg(\bb{1}_{\{\lambda_1=\frac{1}{2}\}},...,\bb{1}_{\{\lambda_k=\frac{1}{2}\}}\bigg)T'\\
&=\sum\limits_{j=1}^k\bb{1}_{\{\lambda_j=\frac{1}{2}\}}Te_j'e_jT'\hspace{11cm}\qed
\end{align*} 

\subsection{Proof of Theorem 5:} Since $B$ is symmetric, by Theorem 3, there exist i.i.d. $\mathcal{N}(0,I)$ random vectors $Y_1,...,Y_n,...$ (possibly in some different probability space where the entire sequence $\{S_n\}_{n\geq0}$ is redefined without changing its distribution) such that \eqref{d} holds.\\

Let $\hat{Y}_j:=Y_jT$. Then we have
\begin{equation}\label{t41}
S_nT=\sum\limits_{j=1}^n\hat{Y}_j\bigg(\frac{j}{n}\bigg)^{-\Lambda}+\epsilon_n\hspace{0.25cm}a.s.
\end{equation}
Since $T$ is orthogonal, $\hat{Y}_j$'s are also i.i.d. $\mathcal{N}(0,I)$ random vectors.\\

Note that $a_n^{(j)}\epsilon_n^{(j)}\xrightarrow{a.s.}0$ for all $j=1,...,k$. Post-multiplying both sides of \eqref{t41} by the diagonal matrix $\text{diag}(a_n^{(1)},...,a_n^{(k)})$ we get
 \begin{equation}\label{t42}
 S_nT\text{diag}\big(a_n^{(1)},...,a_n^{(k)}\big)=\Xi_n+\hat{\epsilon}_n
\end{equation}
where the $l$'th coordinates of $k$-dimensional random vectors $\Xi_n$ and $\hat{\epsilon}_n$ are respectively\\[0.25cm]
$\Xi_n^{(l)}=a_n^{(l)}\sum\limits_{j=1}^n\big(\frac{j}{n}\big)^{-\lambda_l}\hat{Y}_j^{(l)}$ and $\hat{\epsilon}_n^{(l)}=a_n^{(l)}\epsilon_n^{(l)}$. Notice that $\Xi_n^{(1)},...,\Xi_n^{(k)}$ are independent and $\Xi_n^{(l)}$ is a mean 0 normal random variable with variance $\big(a_n^{(j)}\big)^2\sum\limits_{j=1}^n\big(\frac{j}{n}\big)^{-2\lambda_l}$ which converges to $\frac{1}{1-2\lambda_j}$ or $1$ according as $\lambda_j<\frac{1}{2}$ or $\lambda_j=\frac{1}{2}$ respectively. Now since $\hat{\epsilon}_n\xrightarrow{a.s.}0$, using Slutsky's theorem in \eqref{t42}, the proof is complete.\qed\\\\

Before proving Theorem 6 we need the following useful Lemma:
\begin{lemma}
Suppose $b_n$ is a sequence of positive real numbers such that $\sum\limits_{n=1}^{\infty}b_n^2=\infty$ and $\{Z(s)\}_{s\geq0}$ is a standard Brownian motion. Then there exists another standard Brownian motion $\{W(s)\}_{s\geq0}$ on the same probability space where $\{Z(s)\}_{s\geq0}$ is defined, such that 
\begin{equation}\label{l4}
\sum\limits_{j=1}^nb_j\big(Z(j)-Z(j-1)\big)=W\left(\sum\limits_{j=1}^nb_j^2\right)\end{equation}
\end{lemma}  

\begin{proof}
Let $v_0=0$ and $v_n=\sum\limits_{j=1}^nb_j^2$ for $n\geq1$. We define the process $\{W(s)\}_{s\geq0}$ as follows:\\
$W(s):=b_1Z(b_1^{-2}s)$ if $s\in[0,v_1]$ and $$W(s):=\sum\limits_{j=1}^{k-1}b_j\left(Z(j)-Z(j-1)\right)+b_k\left(Z\left(b_k^{-2}(s-v_{k-1})+k-1\right)-Z(k-1)\right)$$ if $s\in(v_{k-1},v_k]$ and $k\geq2$. It can be checked that $\{W(s)\}_{s\geq0}$ is a standard Brownian motion. Choosing $k=n+1$ and $s=v_{k-1}$ we immediately get \eqref{l4}.
\end{proof}

\subsection{Proof of Theorem 6:} By Theorem 3 there exist (possibly in some different probability space where the entire sequence $S_n$ is redefined without changing it's distribution) i.i.d. $\mathcal{N}(0,I)$ random vectors $Y_1,...,Y_n,...$ such that \eqref{d} holds.\\ Without loss of generality suppose the underlying probability space is $(\Omega,\mathcal{F},\Q)$. Consider another probability space $(\Omega_1,\mathcal{F}_1,\mathbb{P}_1)$ where another collection of i.i.d. $\mathcal{N}(0,I)$ normal random vectors $\{Y_d~|~d\in\mathcal{D}\}$ with index set $\mathcal{D}=\{\frac{m}{2^n}~|~m,n\geq1\}\setminus\mathbb{N}$ is defined.\\[0.25cm]
With a slight abuse of notation note that the collection $\{Y_d~|~d\in\mathcal{D}\cup\mathbb{N}\}$ is well defined along with the sequence $S_n$ (without changing it's distribution) on the product space $(\Omega\times\Omega_1,F\otimes\mathcal{F}_1,\Q\otimes\mathbb{P}_1)$ such that $Y_d$'s are i.i.d. $\mathcal{N}(0,I)$ random vectors and \eqref{d} holds for $\{Y_d~|~d\in\mathbb{N}\}$.\\[0.25cm] 
Without loss we will denote the probability space $(\Omega\times\Omega_1,F\otimes\mathcal{F}_1,\Q\otimes\mathbb{P}_1)$ by $(\Omega,\mathcal{F},\Q)$ for simplicity. By Paul L´evy’s construction of Brownian motion (see \cite{Brownian} for example) there exist independent standard 1-dimensional Brownian motions $\{Z^{(1)}(s)\}_{s\geq0},...,\{Z^{(k)}(s)\}_{s\geq0}$ on the same probability space such that $Z^{(l)}(n)=Y^{(l)}_n+Z^{(l)}(n-1)$ for $n=1,2,3,...$ and $l=1,...,k$.\\[0.25cm]
Hence by Lemma 5, on the same probability space we can get another sequence of independent 1-dimensional Brownian motions $\{W^{(1)}(s)\}_{s\geq0},...,\{W^{(k)}(s)\}_{s\geq0}$ such that $$\sum\limits_{j=1}^n{j^{-\lambda_p}}Y_j^{(l)}=W^{(l)}\left(\sum\limits_{j=1}^n{j^{-2\lambda_p}}\right)$$
for all $p,l=1,...,k$. Now from \eqref{d} of Theorem 3 we have
\begin{align*}
S_nT-\epsilon_n &=\sum\limits_{j=1}^n Y_jT\left(\frac{j}{n}\right)^{-\Lambda}\\
&=\left(\sum\limits_{l=1}^kt_{l,1}n^{\lambda_1}\sum\limits_{j=1}^nj^{-\lambda_1}Y_j^{(l)},...,\sum\limits_{l=1}^kt_{l,k}n^{\lambda_k}\sum\limits_{j=1}^nj^{-\lambda_k}Y_j^{(l)}\right)\\
&=\left(G_n^{(1)},...,G_n^{(k)}\right)\hspace{10.5cm}\qed\\
\end{align*} 

\subsection*{Proof of Corollary 2:} By Theorem 6 we can get (possibly in some different probability space where the entire sequence $S_n$ is redefined without changing its distribution) standard 1-dimensional Brownian motions $\{W^{(1)}(s)\}_{s\geq0},...,\{W^{(k)}(s)\}_{s\geq0}$ such that \eqref{emb1} holds. Note that $\sum\limits_{l=1}^kt_{l,1}W^{(l)}(s),...,\sum\limits_{l=1}^kt_{l,k}W^{(l)}(s)$ are independent Brownian motions since $T$ is orthogonal. Without loss denote $\sum\limits_{l=1}^kt_{l,j}W^{(l)}$ by $W^{(j)}$ by a slight abuse of notation. Now the Corollary follows directly from Theorem 6.\qed\\

Now we shall proceed to prove the results in section 3.1.2. Recall the definitions of $k_1,k_2,\tk,\ts_n,\tm_n,\ti,\ti_{k_1},\ti_{k_2}$ and the contexts when they are defined.

\subsection*{Proof of Theorem 7:} We shall proceed as in the proof of Theorem 3. Let $\tilde{Z}_n=\frac{\ts_n}{n}$. Then we have
\begin{equation}\label{d3}
\tilde{Z}_{n+1}=\tilde{Z}_n\left(\frac{n}{n+1}I_{\tk}+\frac{\tl}{n+1}\right)+\frac{\tm_{n+1}}{n+1}
\end{equation} 
and hence
\begin{equation}\label{d4}
\tilde{Z}_n=\tx_1\tc{1}+\sum\limits_{j=2}^n\tm_j\frac{\tc{j}}{j}
\end{equation}
where $\tc{j}$ is defined as in \eqref{C} with $B$ replaced by $\tl$ and $I$ replaced by $I_{\tk}$. Let $\tilde{\Sigma}_n=\sum\limits_{j=1}^n\mathbb{E}[\tm'_{j+1}\tm_{j+1}~|~\F{j}]$. Then $\tilde{\Sigma}_n=\ti'T'\Sigma_nT\ti$. From \eqref{12} we have
 
\begin{align*}
\lVert\Sigma_n-nI\rVert&=\max\limits_{1\leq l\leq k}\sum\limits_{j=1}^n\left(\frac{S_j}{j}Be_l'\right)^2\\
&=\max\limits_{1\leq l\leq k}\sum\limits_{j=1}^n\left(\frac{\hat{S}_j}{j}\Lambda T^{-1}e_l'\right)^2\\
&=O\left(\sum\limits_{j=1}^n\left\lVert\frac{\hat{S}_j}{j}\right\rVert^2\right)\\
&=O(n^{1-\theta_0})
\end{align*}
where $\theta_0=2(1-\eta)\in(0,1)$. Last equality follows because $\frac{\hat{S}_j}{j}=O(j^{\eta-1})$ by Theorem 1.  Hence $\lVert\tilde{\Sigma}_n-\ti'T'T\ti\rVert\leq\lVert T\ti\rVert^2\lVert\Sigma_n-nI\rVert=O(n^{1-\theta_0})$. Also note that condition \eqref{17} is satisfied for $\tm_n$ trivially. Hence by Theorem 1.3 of \cite{zhang2004strong} there exists (possibly in a different probability space where the entire ERWG process $S_n$ is redefined without changing it's distribution) i.i.d. $\mathcal{N}_{\tk}(0,I_{\tk})$ random vectors $Y_1,....,Y_n,...$ such that 
\begin{equation}\label{d5}
\sum\limits_{j=2}^n\tm_j=\sum\limits_{j=2}^nY_j(\ti'T'T\ti)^{\frac{1}{2}}+\tilde{R}_n~a.s.
\end{equation}
where $\tilde{R}_n=O(n^{\frac{1}{2}-\delta'})~a.s.$ for some $\delta'>0$. Now following the proof of Theorem 2, we can get an identity similar to \eqref{cp2}, that is

\begin{equation}\label{d6}
\frac{\ts_n}{n}=\sum\limits_{j=1}^n\frac{Y_j}{j}(\ti'T'T\ti)^{\frac{1}{2}}\left(\frac{j}{n}\right)^{I_{\tk}-\tl}+\tilde{U}_n+\tilde{U}''_n+(\tx_1-Y_1(\ti'T'T\ti)^{\frac{1}{2}})\tc{1}
\end{equation}
where 
\begin{equation}\label{d7}
\begin{split}
\tilde{U}_n&=\sum\limits_{l=2}^{n-1}\tilde{R}_l\left(\frac{\tc{l}}{l}-\frac{\tc{l+1}}{l+1}\right)+\frac{1}{n}\tilde{R}_n\\
&=\sum\limits_{l=2}^{n-1}\tilde{R}_l\left(\frac{\tl}{l(l+1)}\right)\tc{l+1}+\frac{1}{n}\tilde{R}_n
\end{split}
\end{equation}
and 
\begin{equation}\label{d8}
\tilde{U}_n''=\sum\limits_{j=1}^n\frac{Y_j}{j}\ti'T'T\ti\left(\tc{j}-\left(\frac{j}{n}\right)^{I_{\tk}-\tl}\right)
\end{equation}
proceeding as in \eqref{33} we get
\begin{equation}\label{d9}
\tilde{U}_n=(O(n^{-\frac{1}{2}-\delta_1}),...,O(n^{-\frac{1}{2}-\delta_{\tk}}))
\end{equation}
where $0<\delta_j<\min\{\delta',\frac{1}{2}-\lambda_j,\frac{1}{2}\}$ if $1\leq j\leq k_1$ and $\delta_j=0$ if $k_1+1\leq j\leq\tk$. Similarly it can be checked that an analogue of \eqref{as} holds for $\tilde{U}_n''$, that is
\begin{equation}\label{d10}
\tilde{U}_n''=(O(n^{\lambda_1-1}),...,O(n^{\lambda_{\tk}-1}))
\end{equation}
finally notice that the last part of \eqref{d6} is also $(O(n^{\lambda_1-1}),...,O(n^{\lambda_{\tk}-1}))$, using this along with \eqref{d9} and \eqref{d10} in \eqref{d6} we get
\begin{equation}\label{d11}
\frac{\ts_n}{n}=\sum\limits_{j=1}^nY_j(\ti'T'T\ti)^{\frac{1}{2}}\frac{1}{j}\left(\frac{j}{n}\right)^{I_{\tk}-\tl}+\frac{\tilde{\epsilon}_n}{n}~a.s.
\end{equation}
where $\tilde{\epsilon}_n$ is defined in \eqref{d2}. Finally the Theorem follows by multiplying both sides of \eqref{d11} by $n$.\qed\\

\subsection*{Proof of Corollary 3:} Define $a_n$ to be $n^{-\frac{1}{2}}$ if $\max\limits_{j\leq\tk}\lambda_j<\frac{1}{2}$ and $(n\log n)^{-\frac{1}{2}}$ if $\max\limits_{j\leq\tk}\lambda_j=\frac{1}{2}$. Then by Theorem 7, we have $a_n\ts_n=\tilde{\Theta}_n+a_n\tilde{\epsilon}_n~a.s.$ (on some probability space where the entire sequence $S_n$ is redefined without changing it's distribution) where $a_n\tilde{\epsilon}_n\xrightarrow{a.s.}0$ and $\tilde{\Theta}_n$ is a $\tk$-dimensional normal vector with mean $0$ and covariance matrix $\text{Var}(\tilde{\Theta}_n)$ given by 
$$\text{Var}(\tilde{\Theta}_n)_{p,q}=(\ti'T'T\ti)_{p,q}a_n^2\sum\limits_{j=1}^n\left(\frac{j}{n}\right)^{-\lambda_p-\lambda_q}$$ 
For all $p,q=1,...,\tk$. Note that $\sum\limits_{j=1}^n(\frac{j}{n})^{-\lambda_p-\lambda_q}\sim n(1-\lambda_p-\lambda_q)^{-1}$ if $\lambda_p+\lambda_q<\frac{1}{2}$ and $\sum\limits_{j=1}^n(\frac{j}{n})^{-\lambda_p-\lambda_q}\sim n\log n$ if $\lambda_p+\lambda_q=\frac{1}{2}$. Hence $\text{Var}(\tilde{\Theta}_n)\rightarrow\tilde{\Sigma}^{(1)}$ or $\tilde{\Sigma}^{(2)}$ respectively according as $\max\limits_{1\leq j\leq\tk}\lambda_j<\frac{1}{2}$ or $=\frac{1}{2}$. Now the corollary follows by Slutsky's theorem\qed\\

\subsection*{Proof of Corollary 4:} Notice that $\ts_{n,1}=S_nT\ti_{k_1}$. Define $\tilde{Z}_{n,1}=\frac{\ts_{n,1}}{n}$, $\tilde{\Lambda}_{1}=\ti_{k_1}^{'}\Lambda\ti_{k_1}$ and $\Delta\tilde{M}_{n,1}=\Delta M_nT\ti_{k_1}$. It can be checked easily that similar identity like \eqref{d3} holds, that is
\begin{equation}\label{cor42}
\tilde{Z}_{n+1,1}=\tilde{Z}_{n,1}\left(\frac{n}{n+1}I_{k_1}+\frac{\tl_{1}}{n+1}\right)+\frac{\tm_{n+1,1}}{n+1}
\end{equation} 
Now proceeding as in the proof of Theorem 7, we can get i.i.d. $\mathcal{N}_{k_1}(0,I_{k_1})$ random vectors $Y_1,Y_2,...,Y_n,...$(possibly in some different probability space where the sequence $S_n$ is defined with it's distribution unchanged) such that

\begin{equation}\label{cor43}
\ts_{n,1}=\sum\limits_{j=1}^{n}Y_j(\ti_{k_1}^{'}T'T\ti_{k_1})^{\frac{1}{2}}\left(\frac{j}{n}\right)^{-\tl_{1}}+\beta_n~a.s.
\end{equation}
$\beta_n^{(j)}=O(n^{\frac{1}{2}-\delta_j})$ for $j=1,...,k_1$, where $\delta_j$'s are all positive. Now we can mimic the proof of Corollary 3 and since the case $\lambda_p+\lambda_q=\frac{1}{2}$ can not appear, \eqref{cor41} follows.\qed\\

\subsection*{Proof of Theorem 8:} First notice that we can write \eqref{cor42} as the following stochastic approximation scheme:
\begin{equation}
\label{t71}
\tilde{Z}_{n+1,1}=\tilde{Z}_{n,1}+\gamma_n[h(\tilde{Z}_{n,1})+r_{n+1}]+\sigma_n\varepsilon_{n+1}
\end{equation}
where $\gamma_n=\sigma_n=\frac{1}{n+1}$, $r_{n+1}=0$, $\varepsilon_{n+1}=\Delta\tilde{M}_{n+1,1}$ and the function $h:\mathbb{R}^{k_1}\rightarrow\mathbb{R}^{k_1}$ is defined by $h(x)=-x(I_{k_1}-\Lambda_1)$ for all $x\in\mathbb{R}^{k_1}$.\\[0.25cm]
Our goal is to use Theorem 1 of \cite{Mokkadem}.  We just need to verify all the conditions (A1)-(A5) stated in \cite{Mokkadem}.\\[0.25cm]
First Let $v_n=\frac{\gamma_n}{\sigma_n^2}=n+1$  and $s_n=\sum\limits_{j=1}^n\gamma_j$. Note that (A1) holds with $z^*=0$ by Corollary 1. The Jacobian matrix $H$ of $h$ is $\Lambda_1-I_{k_1}$ and $\max\{\Re(\lambda)~|~\lambda\in Sp(H)\}=-L<0$ where $L=1-\max\limits_{1\leq j\leq k_1}\lambda_j$ which ensures (A2).\\[0.25cm]
By definition $\gamma_n=\sigma_n\rightarrow0$ and $\frac{\gamma_n}{\gamma_{n+1}}=1+\frac{1}{n+1}=1+O(\gamma_n)$ , further $\frac{1}{\gamma_n}[1-\frac{v_{n-1}}{v_n}]=\xi=1\in[0,2L)$ since $\max\limits_{1\leq j\leq k_1}<\frac{1}{2}$, which ensures (A3). Notice that (A5) is satisfied trivially since $r_n=0$.\\[0.25cm]
Now it remains to check (A4) only. Proceeding as in the proof of Theorem 7, we can get i.i.d. $\mathcal{N}_{k_1}(0,I_{k_1})$ random vectors $Y_1,Y_2,...,Y_n,...$(possibly in some different probability space where the sequence $S_n$ is redefined with it's distribution unchanged) such that \eqref{d5} holds with $\tm_{j},\ti$ and $\tilde{R}_n$ replaced by respectively $\tm_{j,1},\ti_{k_1}$ and some suitable $\tilde{R}_{n,1}=O(n^{\frac{1}{2}-\delta'})$ for some $\delta'>0$. Since $n^{\frac{1}{2}-\delta'}=o\left((n+1)\log\log n\right)$, we have
\begin{equation}
\label{t72}
\left\lVert \sum\limits_{j=2}^n\tm_{j,1}-Y_j(\ti_{k_1}^{'}T'T\ti_{k_1})\right\rVert=o\left(\sqrt{\frac{\log s_n}{\gamma_n}}\right)\hspace{0.25cm}a.s.
\end{equation}
Note that $\ti_{k_1}^{'}T'T\ti_{k_1}$ is positive definite, hence condition (A4) is also satisfied.\\[0.25cm]
Hence by Theorem 1 of \cite{Mokkadem} we conclude that with probability 1 the sequence 
$$\left(\frac
{n+1}{2\log s_n}\right)^{\frac{1}{2}}\tilde{Z}_{n,1}$$ 
is relatively compact and it's set of limit points is the ellipsoid $C=\big\{x\in\mathbb{R}^{k_1}~|~x\Sigma^{-1}x'\leq1\big\}$, where $\Sigma=\int_0^{\infty}e^{(\Lambda_1-\frac{1}{2}I_{k_1})t}(\ti_{k_1}^{'}T'T\ti_{k_1})e^{(\Lambda_1-\frac{1}{2}I_{k_1})t}dt=\tilde{\Sigma}_1^*$. Now the Theorem follows by noticing that $s_n\sim\log n$.\qed\\

\subsection*{Proof of Theorem 9:} Notice that $\ts_{n,2}=S_n T\ti_{k_2}$. Defining $\tilde{Z}_{n,2}=\frac{\ts_{n,2}}{n},\tilde{\Lambda}_2=\ti_{k_2}^{'}\Lambda\ti_{k_2}=\frac{1}{2}I_{k_2}$ and $\tm_{n,2}=\Delta M_n T\ti_{k_2}$ and proceeding as in the proof of Theorem 7 we can get i.i.d. $\mathcal{N}_{k_2}(0,I_{k_2})$ random vectors $Y_1,Y_2,...,Y_n,...$(possibly in some different probability space where the sequence $S_n$ is redefined with it's distribution unchanged)

\begin{equation}
\label{t81}
\ts_{n,2}=\sum\limits_{j=1}^{n}Y_j(\ti_{k_2}^{'}T'T\ti_{k_2})^{\frac{1}{2}}\left(\frac{j}{n}\right)^{-\tl_2}+O(n^{\frac{1}{2}})\hspace{0.25cm}a.s.
\end{equation}
Define $\Gamma=\ti_{k_2}^{'}T'T\ti_{k_2}$, $V_j=j^{-\frac{1}{2}}Y_j\Gamma$ and $B_n=\sum\limits_{j=1}^n\mathbb{E}[V_j'V_j]=h_n\Gamma$, where $h_n=\sum\limits_{j=1}^n\frac{1}{j}$.\\[0.25cm] Note that the sum in \eqref{t81} can be expressed as $n^{\frac{1}{2}}\sum\limits_{j=1}^nV_j$, since $(\frac{j}{n})^{-\tl_2}=(\frac{j}{n})^{-\frac{1}{2}}I_{k_2}$.\\[0.25cm]
Our goal is to use Corollary 1 of \cite{Mokkadem2} on the sum of independent Gaussian vectors $\sum\limits_{j=1}^nV_j$. Note that (A1) of \cite{Mokkadem2} is satisfied trivially. It only remains to check the conditions (i),(ii) and (iii) of Corollary 1 of \cite{Mokkadem2}.\\[0.25cm]
Since $\Gamma$ is positive definite and $\rho_{\text{min}}(B_n)=\rho_{\text{min}}(\Gamma)h_n\sim\rho_{\text{min}}(\Gamma)\log n$ and $\lVert B_n\rVert=\rho_{\text{max}}(B_n)=\rho_{\text{max}}(\Gamma)h_n\sim\rho_{\text{max}}(\Gamma)\log n$, it is immediate that $\rho_{\text{min}}(B_n)\rightarrow\infty$ and $\rho_{\text{min}}(B_{n})^{-1}\rho_{\text{min}}(B_{n-1})\rightarrow1$.\\[0.25cm]
Further since $\log\log \rho_{\text{min}}(B_{n})\sim\log\log\log n$ and $\log\log \rho_{\text{max}}(B_{n})\sim\log\log\log n$, we have $(\log\log \rho_{\text{max}}(B_{n}))^{-1}\log\log \rho_{\text{min}}(B_{n})\rightarrow1$.\\[0.25cm]
Finally notice that $\lVert B_n^{-1}B_m\rVert=h_n^{-1}h_m=\lVert B_m^{-1}\rVert^{-1}\lVert B_n^{-1}\rVert$. Hence all conditions (i),(ii) and (iii) of Corollary 1 of \cite{Mokkadem2} is satisfied. So we conclude that with probability 1, the sequence 
$$\Gamma^{-\frac{1}{2}}(2h_n\log\log\lVert B_n\rVert)^{-\frac{1}{2}}\sum\limits_{j=1}^nV_j$$ is relatively compact and its set of limit points is the closed unit ball in $\mathbb{R}^{k_2}$. Note that 

\begin{equation}
\label{t82}
(2nh_n\log\log\lVert B_n\rVert)^{-\frac{1}{2}}\ts_{n,2}=\Gamma^{\frac{1}{2}}\left(\Gamma^{-\frac{1}{2}}(2h_n\log\log\lVert B_n\rVert)^{-\frac{1}{2}}\sum\limits_{j=1}^nV_j\right)+\tilde{\alpha}_n
\end{equation}
where $\tilde{\alpha}_n=O((h_n\log\log\lVert B_n\rVert)^{-\frac{1}{2}})\hspace{0.25cm}a.s.$ Now using the fact $h_n\log\log\lVert B_n\rVert\sim\log n\log\log\log n$ in \eqref{t82} the proof is complete.\qed\\

Proof of the results in section 3.2., that is for general memory matrix $B$, relies solely on the properties of stochastic approximation scheme \eqref{2.2.1}. 

\subsection*{Proof of theorem 10:} We will first show that $Z_n=\frac{S_n}{n}\xrightarrow{a.s.}0$. Focusing on the stochastic approximation scheme \eqref{2.2.1} and using Theorem 2 of \cite{Borkar} we notice that the set $\Phi$ of all limit points of $Z_n$ is stable by the flow of the ODE $\frac{d}{dt}x(t)=-h(x(t))$, which has the only solution $x(t)=x(0)e^{(B-I)t}$. Since real part of all eigenvalues of $B-I$ are negative ($\eta<1$), we must have $x(t)\rightarrow0$ as $t\rightarrow\infty$, which implies $\Phi=\{0\}$. So we conclude that $Z_n\xrightarrow{a.s.}0$.\\[0.25cm]
Let $\rho:=\min\{\Re(\lambda)~|~\lambda\in Sp(I-B)\}$. Clearly $\rho=1-\eta>0$. In order to prove \eqref{t91} we will use Corollary 2.1 of \cite{Zhang2} on the stochastic approximation scheme \eqref{2.2.1}. For that we just need to verify assumption 2.1 and the condition 2.10 of \cite{Zhang2}. Note that the first one is satisfied trivially since $0$ is an equilibrium point of $\{h=0\}$ and all eigenvalues of the Jacobian $Dh(0)=I-B$ have positive real part. Define $\Sigma_n$ as in the proof of Lemma 4. Recall the expression of $\Sigma_n$ in \eqref{12}. Using the fact that $\frac{S_n}{n}\xrightarrow{a.s.}0$, we conclude that 
\begin{equation}
\label{t92}
\frac{1}{n}\Sigma_n\xrightarrow{a.s.}I
\end{equation}
Hence Assumption 2.1 and condition 2.10 of \cite{Zhang2} are satisfied. Hence by corollary 2.1 of \cite{Zhang2} we have $\frac{S_n}{n}=o(n^{-\frac{1}{2}\wedge\rho+\delta'})\hspace{0.25cm}a.s.$ for all $\delta'>0$. Now the Theorem follows by setting $\delta=\frac{1}{2}-\frac{1}{2}\wedge\rho+\delta'$.\qed\\

Before proceeding further we prove a generalised version of Lemma 4.

\begin{lemma}
Suppose $\eta\leq\frac{1}{2}$ and $B$ is any interacting matrix (not necessarily diagonalisable). Then there exists a probability space where the entire sequence $\{S_n\}_{n\geq0}$ is redefined without changing it's distribution and in the same space there exists a sequence $\{Y_n\}_{n\geq2}$ of i.i.d. $\mathcal{N}(0,I)$ random vectors such that

\begin{equation}\label{l51}
\sum\limits_{j=1}^n\Delta M_{j+1}=\sum\limits_{j=1}^nY_{j+1}+O\big(n^{\frac{1}{2}-\delta}\big)~~~a.s.~~\text{for some}~\delta>0
\end{equation}
\end{lemma}

\begin{proof}
We proceed along the same lines of the proof of Lemma 4. Notice that everything would have proceeded smoothly if \eqref{l31} and \eqref{15} were true. For general $B$ this comes from the fact that $S_n=o(n^{\frac{1}{2}+\delta})~a.s.$ for all $\delta>0$ by Theorem 10.
\end{proof}

\subsection*{Proof of Theorem 11:} As in the proof of Theorem 8, we again will use Theorem 1 of \cite{Mokkadem} on the stochastic approximation scheme \eqref{2.2.1}. To avoid confusion we rewrite \eqref{2.2.1} (with a slight abuse of notation) as follows:
\begin{equation}
\label{t101}
Z_{n+1}=Z_n+\gamma_n[h(Z_n)+r_{n+1}]+\sigma_n\varepsilon_{n+1}
\end{equation}
where $h(x)=x(B-I)$ for all $x\in\mathbb{R}^k$. $\varepsilon_{n}=\Delta M_n$ and $r_n=0$. Further $\gamma_n=\sigma_n=\frac{1}{n+1}$. Now we just need to check the conditions (A1)-(A5) of \cite{Mokkadem}. Note that (A1) follows directly from Theorem 10 with $z^*=0$. (A2) follows with $L=1-\eta$. Validity of (A3)(i) is already checked in the proof of Theorem 8 and it can be checked that (A3)(ii) holds with $\xi=1$. Further (A5) holds trivially since $r_n=0$. Finally we notice that (A4) follows from \eqref{l51} with $\Gamma=I$. Hence by Theorem 1 of \cite{Mokkadem} we conclude that with probability 1 the sequence $\left(\frac{n+1}{2\log s_n}\right)^{\frac{1}{2}}Z_n$ is relatively compact and its set of limit points is the ellipsoid $E$. The proof is complete after noticing that $s_n=\sum\limits_{j=1}^n\frac{1}{j+1}\sim\log n$.\qed\\

Before going into the proof of Theorem 12, recall the definitions of $\h,\T$ and $\nu$.

\subsection*{Proof of Theorem 12:} First consider the case $\eta<\frac{1}{2}$. we will use Theorem 1.1 of \cite{Zhang2} on the stochastic approximation scheme \eqref{2.2.1}. Note that the martingale differences $\Delta M_n$ are uniformly bounded by some non-random constant and
\begin{equation}
\label{t113}
\begin{split}\mathbb{E}[\Delta M'_{n+1}\Delta M_{n+1}~|~\F{n}] &=\text{diag}\left(1-\left(\frac{S_n}{n}Be_1'\right)^2,...,1-\left(\frac{S_n}{n}Be_k'\right)^2\right)\\
&\xrightarrow{a.s.}I\hspace{1cm}\text{(by Theorem 10)}
\end{split}
\end{equation}
Further notice that $\rho=1-\eta>\frac{1}{2}$. Hence all conditions of Theorem 1.1 of \cite{Zhang2} are satisfied for \eqref{2.2.1}. So we conclude that \eqref{t111} holds. Now assume that $\eta=\frac{1}{2}$. Note that $0$ is an equilibrium point of $\{h=0\}$ and all eigenvalues of $\h$ have positive real parts. Hence assumption 2.1 of \cite{Zhang2} holds. Further since $h$ is linear and $h(0)=0$ we note that assumption 2.2 of \cite{Zhang2} holds trivially. Finally assumption 2.3 follows from uniform boundedness of $\Delta M_n$'s and from \eqref{t92}. Hence using Theorem 2.1 of \cite{Zhang2} we see that \eqref{t112} holds.\qed\\

\subsection*{Proof of Theorem 13:} We will use Theorem 3.1 of \cite{Zhang2}. Note that $Z_n\xrightarrow{a.s.}0$ from Theorem 10 and $h$ is linear with $h(0)=0$ such that all eigenvalues of $\h$ have positive real parts (since $\eta\leq\frac{1}{2}$). Hence assumption 2.1 and 2.2 of \cite{Zhang2} are satisfied. Further note that condition (3.3) and (3.4) of \cite{Zhang2} are satisfied trivially for any $\epsilon_0\in(0,1)$. It only remains to show that condition (3.5) of \cite{Zhang2} holds. Note that $\frac{S_n}{n}=o(n^{-\epsilon_0})~a.s.$ for all $\epsilon_0\in(0,\frac{1}{2})$ from Theorem 10. Using this fact in \eqref{t113} we get
\begin{equation}
\label{t121}
I-\mathbb{E}[\Delta M'_{n+1}\Delta M_{n+1}~|~\F{n}]=o(n^{-\epsilon_0})\hspace{0.25cm}a.s.
\end{equation}
Summing over first $n$ terms in \eqref{t121} and using Proposition 1 we get
\begin{equation}
\label{t122}
nI-\sum\limits_{j=1}^n\mathbb{E}[\Delta M'_{j+1}\Delta M_{j+1}~|~\F{j}]=o\left(\sum\limits_{j=1}^n j^{-\epsilon_0}\right)=o(n^{1-\epsilon_0})\hspace{0.25cm}a.s.
\end{equation}
and hence condition (3.5) of \cite{Zhang2} is also satisfied. So by Theorem 3.1 of \cite{Zhang2} and noticing that $(\frac{n}{s})(\frac{s}{n})^{I-B}=(\frac{s}{n})^{-I}(\frac{s}{n})^{I-B}=(\frac{s}{n})^{-B}$ (since $I$ and $I-B$ communicates), the proof is complete.\qed\\

\subsection*{Proof of Theorem 13:} This comes directly by applying Theorem 2.2 of \cite{Zhang2} on \eqref{2.2.1}. Notice that assumption 2.1 and 2.2 of \cite{Zhang2} still holds. Further \eqref{t92} ensures that assumption 2.10 of \cite{Zhang2} also holds.\qed

\section{Proof of the results in section 4}
Most of the results in section 4 are just direct application of general results in section 3 except for Theorem 15 and 16, where we shall deal with the situation $\eta=1$ and see how the limit random variable look like in this case.

\subsection*{Proof of results in diffusive regime:}  Note that $p\in(\frac{1}{4},\frac{3}{4})$ and hence $\eta=|2p-1|<\frac{1}{2}$. Hence the IERG $S_n$ is in globally diffusive regime.

\subsubsection*{\textbf{Proof of Corollary 5:}} A direct application of Theorem 11 yields this corollary. Note that since $B$ is symmetric $\Sigma^{(1)}=(I-2B)^{-1}$ from the last part of Theorem 4 and hence $(x,y)(\Sigma^{(1)})^{-1}(x,y)'=x^2-4(2p-1)xy+y^2$.\qed

\subsubsection*{\textbf{Proof of Corollary 6:}} This comes directly from the first part of Theorem 12.\qed

\subsubsection*{\textbf{Proof of Corollary 7:}}
This comes from Theorem 1, since both projected walks $\st{n}{1}$ and $\st{n}{2}$ are diffusive.\qed

\subsubsection*{\textbf{Proof of Corollary 8:}} We apply Corollary 2 since the ERWG $S_n$ is symmetric.\qed\\

\subsection*{Proof of the results in critical regime:} Since $p\in\{\frac{1}{4},\frac{3}{4}\}$,  we have $\eta=\frac{1}{2}$ ,i.e., the ERWG $S_n$ is in globally critical regime.

\subsubsection*{\textbf{Proof of Corollary 9:}} 
We will only prove part (a) since proof of part (b) is similar. Let $p=\frac{3}{4}$. Then $\lambda_1=\frac{1}{2}$ and $\lambda_2=-\frac{1}{2}$. Hence by Theorem 2, we conclude that th probability 1, the sequences $\frac{\st{n}{1}}{\sqrt{2n\log n\log\log\log n}}$ and $\frac{\st{n}{2}}{2n\log\log n}$ are relatively compact with sets of limit points $[-\sigma(1/2),\sigma(1/2)]$ and $[-\sigma(-1/2),\sigma(-1/2)]$ respectively. Note that $T^2=I$, hence $\sigma^2(1/2)=e_1T^2e_1'=1$ and $\sigma^2(-\frac{1}{2})=2^{-1}e_2T^2e_2'=\frac{1}{2}$. This proves the first part of (a).\\[0.25cm]
Further notice that with probability 1, the joint sequence $\frac{\hat{S}_n}{\sqrt{2n\log n\log\log\log n}}=\frac{(\st{n}{1},\st{n}{2})}{\sqrt{2n\log n\log\log\log n}}$ has its set of limit points $M=[-1,1]\times\{0\}$ which implies that the same for $\frac{S_n}{\sqrt{2n\log n\log\log\log n}}$ is $MT^{-1}=MT=\big\{(x,x)~|~|x|\leq\frac{1}{\sqrt{2}}\big\}$.\qed

\subsubsection*{\textbf{Proof of Corollary 10:}}
We will prove only part (a). Proof of (b) is similar. Assume $p=\frac{3}{4}$. Since $B$ is symmetric and $\lambda_1=-\lambda_2=\frac{1}{2}$, applying Theorem 5 we get 
$$\left(\frac{\st{n}{1}}{\sqrt{n\log n}},\frac{\st{n}{2}}{\sqrt{n}}\right)=S_nT\text{diag}\big(n^{-\frac{1}{2}},(n\log n)^{-\frac{1}{2}}\big)\xrightarrow{d}\mathcal{N}_2\left((0,0),\begin{bmatrix}
1 & 0 \\
0 & \frac{1}{2}
\end{bmatrix}
\right)$$
which proves \eqref{cor101}. It is easy to see from Cram\'{e}r-Wold theorem that 
$$
\frac{S_nT}{\sqrt{n\log n}}\xrightarrow{d}(Z,0)
$$
where $Z$ is a 1-dimensional standard normal variable. Hence $\frac{S_n}{\sqrt{n\log n}}$ converges in distribution to $(Z,0)T=\frac{1}{\sqrt{2}}(Z,Z)$ by continuous mapping theorem.\qed

\subsubsection*{\textbf{Proof of Corollary 11:}} Follows directly from the moments part of Theorem 1.\qed

\subsubsection*{\textbf{Proof of Corollary 12:}} Follows directly from Corollary 2.\qed\\

\subsection*{Proof of the results in super-diffusive regime:} First let us focus on the case $p\in(\frac{3}{4},1]$, i.e , when the elephants are in high memory regime. Recall the eigenvalues of the memory matrix $B$ are $\lambda_1=2p-1>\frac{1}{2}$ and $\lambda_2=1-2p<\frac{1}{2}$. Their corresponding projected walks are $\st{n}{1}=\frac{1}{\sqrt{2}}(\s{n}{1}+\s{n}{2})$ and $\st{n}{2}=\frac{1}{\sqrt{2}}(\s{n}{1}-\s{n}{2})$ respectively. Note that $\st{n}{1}$ is super-diffusive but $\st{n}{2}$ is diffusive.

\subsubsection*{\textbf{Proof of Corollary 13:}} From \eqref{t11} of Theorem 1, choosing $j=1$, we get 
$$(d_n^{(1)})^{-1}\st{n}{1}\xrightarrow{a.s.,~L^m}\st{\infty}{1}$$
for all $m>0$. Where the random variable $\st{\infty}{1}$ satisfies $\mathbb{E}[\st{\infty}{1}]=\mathbb{E}[\st{2}{1}]$. It is easy to see that $\mathbb{E}[\st{2}{1}]=2\sqrt{2}p(q_1+q_2-1)$. Further assuming $p\in(\frac{3}{4},1)$, i.e., $\eta<1$, we get $\text{Var}(\st{\infty}{1})>0$ from Theorem 1.\\
For $p=1$ (i.e. $\eta=1$), the fact that $\text{Var}(\st{\infty}{1})>0$ will be proved in the proof of Theorem 15. Now since $\st{n}{2}$ is diffusive, again by Theorem 1, we have 
$$\st{n}{1}=o\big(n^{\frac{1}{2}}(\log n)^{\frac{1}{2}+\epsilon}\big)=o(d_n^{(1)})\hspace{0.25cm}a.s.$$
since $d_n^{(1)}\sim\frac{n^{2p-1}}{\Gamma(2p+1)}$ and $2p-1>\frac{1}{2}$. Further note that $\mathbb{E}[|\st{n}{2}|^m]=O(n^{\frac{m}{2}})=o(n^{m(2p-1)})$ for all $m>0$. Hence 
$$(d_n^{(1)})^{-1}\st{n}{2}\xrightarrow{a.s.,~L^m}0$$
for all $m>0$. So we have for $j=1,2$
$$(d_n^{(1)})^{-1}\s{n}{j}=(\sqrt{2}d_n^{(1)})^{-1}\big(\st{n}{1}+(3-2j)\st{n}{2}\big)\xrightarrow{a.s.,~L^m}\frac{1}{\sqrt{2}}\st{\infty}{1}$$ 
for all $m>0$. Now the corollary is proved by setting $S_{p,q_1,q_2}=\frac{1}{\sqrt{2}}\st{\infty}{1}$.\qed

\subsubsection*{\textbf{Proof of Theorem 15}:} Recall that $S_n\sim\text{ERWG}(p,q_1,q_2,\G)$ where $\G$ is the directed graph with vertices 1 and 2 and directed edges $(1,2)$ and $(2,1)$ respectively. More precisely $S_n=\sum\limits_{j=1}^nX_j$ where $\X{1}{j}$ has Rad($q_j$) distribution and $X_{n+1}$ is recursively defined to be $\left(\yi{n}{1}\X{\D{n}{1}}{2},\yi{n}{2}\X{\D{n}{2}}{1}\right)$. Where $\yi{n}{1},\yi{n}{2}$ and $\D{n}{1},\D{n}{2}$ are all independent, respectively have Rad($p$) and Unif$(\{1,...,n\})$ distributions and all are independent of $\X{1}{1},\X{1}{2}$. Define $X_n'=-X_n$ and $S_n'=\sum\limits_{j=1}^nX_j'$, then by definition we have $-S_n=S_n'\sim\text{ERWG}(p,1-q_1,1-q_2,\G)$. This proves \eqref{t141}.\\[0.25cm]
Now let us focus on the case $p=1$ (i.e. $\eta=1$). First of all notice that $d_n^{(1)}=\frac{n}{2}$ and hence by Corollary 13, we observe that the random variable $S_{1,q_1,q_2}$ is the almost sure limit of $\frac{2\s{n}{j}}{n}$ for $j=1,2$ and for all $m>0$. Notice that, for any $n\geq1$, we have $\s{n}{j}=n$ on the set $[\X{1}{1}=\X{1}{2}=1]$ and $\s{n}{j}=-n$ on $[\X{1}{1}=\X{1}{2}=-1]$, for $j=1,2$. Hence $S_{1,q_1,q_2}=2$ on the set $[\X{1}{1}=\X{1}{2}=1]$ and $S_{1,q_1,q_2}=-2$ on the set $[\X{1}{1}=\X{1}{2}=-1]$.\\
Now it remains to study the limit on the set $[\X{1}{1}\neq\X{1}{2}]$. Notice that the projection martingale corresponding to the eigenvalue $\lambda_1=1$ is $(d_n^{(1)})^{-1}\st{n}{1}=\frac{2\st{n}{1}}{n}$. From the proof of Corollary 13 we know that 
$$\frac{2\st{n}{1}}{n}\xrightarrow{a.s.}\sqrt{2}S_{1,q_1,q_2}$$
in other words
$$\frac{\s{n}{1}+\s{n}{2}}{n}\xrightarrow{a.s.}S_{1,q_1,q_2}$$
Now we shall study the martingale $\frac{\s{n}{1}+\s{n}{2}}{n}$ on the set $[\X{1}{1}\neq\X{1}{2}]$.\\
Without loss, assume (possibly in some different probability space) that $\X{1}{1}=-\X{1}{2}=1$. Since $p=1$, it is immediate that $\X{2}{1}=-\X{2}{2}=-1$. Now for $n\geq1$, define
$$Y_n=\frac{1}{4}(\X{n+2}{1}+\X{n+2}{2})+\frac{1}{2}$$
and for $n\geq0$, define
$$M_n=\frac{1}{4n+8}\big(\s{n+2}{1}+\s{n+2}{2}\big)+\frac{1}{2}$$
Note that $M_0=\frac{1}{2}$ and $(n+2)M_n=1+\sum\limits_{j=1}^nY_j$ for $n\geq1$. Clearly $M_n$ is a martingale and both $M_n,Y_{n+1}$ take values in $[0,1]$. Observe that, by definition
$$M_{n+1}=(1-r_n)M_n+r_nY_{n+1}$$
where $r_n=\frac{1}{n+3}$ for all $n\geq0$. Further we have for all $n\geq0$
$$\mathbb{E}[Y_{n+1}~|~M_0,M_1,...,M_n]=M_n$$
We shall apply Theorem 2.2 of \cite{Polarization} on the martingale $M_n$. Note that 
$$e^{-\sum\limits_{j=0}^nr_j}\sum\limits_{j=0}^nr_j=O\left(\frac{\log n}{n}\right)$$
and $M_0=\frac{1}{2}$, hence by Theorem 2.2 of \cite{Polarization} we conclude that $\Q(0<M_{\infty}<1)=1$ where $M_{\infty}$ is the almost sure limit of $M_n$.\\
Note that the martingale $M_n$ remains unchanged if we start with $\X{1}{1}=-\X{1}{2}=-1$. Further the distribution of $M_{\infty}$ does not depend on the distribution of $(\X{1}{1},\X{1}{2})$ and $q_1,q_2$. Now observe that on the set $[\X{1}{1} \neq \X{1}{2}]$, the distribution of the limiting random variable $S_{1, q_1, q_2}$ is identical to that of $4M_{\infty} - 2$. Without loss we can assume that there is a random variable $U_1$ defined in the same probability space where $(\X{1}{1},\X{1}{2})$ is defined such that $U_1\stackrel{d}{=}2M_{\infty}-1$ and is independent of $(\X{1}{1},\X{1}{2})$. This proves \eqref{t142} along with the fact that $\Q(-1<U_1<1)=1$.\\
Now it remains to show $\mathbb{E}[U_1]=0$ and $\mathbb{E}[U_1^2]>0$. From the above discussion, without loss we can assume that $U_1$ is the almost sure limit of the martingale (possibly defined in some different probability space) $V_n=\frac{1}{2n}\big(\s{n}{1}+\s{n}{2}\big)$. Where the sequence 
$$S_n:=\big(\s{n}{1},\s{n}{2}\big)\sim\text{ERWG}(p=1,q_1=1,q_2=0,\G)$$ 
Let $\delta_{n+1}:=\frac{1}{2}\big(\X{n+1}{1}+\X{n+1}{2}\big)-V_n$ for $n\geq2$. Then $\delta_n$ is a uniformly bounded martingale difference sequence and we can express $V_n$ for $n\geq3$ as follows:
$$V_n=\sum\limits_{j=3}^n\frac{\delta_j}{j}$$
First of all since $V_n$ is uniformly bounded, all moments of it will exist and it is immediate that $\mathbb{E}[V_n]=0$ which proves $\mathbb{E}[U_1]=0$.
Further since it's quadratic variation $\langle V\rangle_n$ converges, we have 
\begin{equation}\label{pt151}
\mathbb{E}[U_1^2]=\lim\limits_{n\rightarrow\infty}\mathbb{E}[V_n^2]=\sum\limits_{n=3}^{\infty}\frac{1}{n^2}\mathbb{E}[\delta_n^2]
\end{equation}
Now it is easy to see that 

\begin{equation}\label{pt152}
4\mathbb{E}\big[\delta_{n+1}^2~|~S_j,j\leq n\big]=2+2\frac{\s{n}{1}\s{n}{2}}{n^2}-\left(\frac{\s{n}{1}+\s{n}{2}}{n}\right)^2
\end{equation}
Note that the projected walk $\frac{1}{\sqrt{2}}\big(\s{n}{1}-\s{n}{2}\big)$ is diffusive, hence by Theorem 1 we have $\frac{1}{\sqrt{2}}\big(\s{n}{1}-\s{n}{2}\big)=o(n)~a.s.$, which implies that both of $\frac{\s{n}{1}}{n}$ and $\frac{\s{n}{2}}{n}$ converge almost surely to $U_1$. Now taking expectation on both sides of \eqref{pt152} and using dominated convergence theorem, we get 
$$
4\mathbb{E}\big[\delta_{n+1}^2\big]\rightarrow2-2\mathbb{E}\big[U_1^2\big]
$$
Since $\Q(-1<U_1<1)>0$, we must have $0<\mathbb{E}[U_1^2]<1$. in other words $\liminf\limits_{n\rightarrow\infty}\mathbb{E}[\delta_n^2]>0$, hence from \eqref{pt151} we have $\mathbb{E}[U_1^2]>0$.\qed

\subsubsection*{\textbf{Proof of Corollary 14:}} Since $\st{n}{1}$ is super-diffusive and $\lambda_1=2p-1\in(\frac{1}{2},1)$, from Theorem 2, we have 
$$
n^{2p-\frac{3}{2}}\big((d_n^{(1)})^{-1}\st{n}{1}-\st{\infty}{1}\big)\xrightarrow{d}\mathcal{N}(0,\sigma^2(\lambda_1))$$
where recall from the proof of Corollary 13 that $\st{\infty}{1}=\sqrt{2}S_{p,q_1,q_2}$ and from Theorem 2 $\sigma^2(\lambda_1)=(2\lambda_1-1)^{-1}\Gamma(\lambda+2)^2e_1T'Te_1'=\frac{\Gamma(2p+1)^2}{4p-3}=:\sigma_1^2$. Further with probability 1, the sequence $n^{2p-\frac{1}{2}}(2\log\log n)^{-\frac{1}{2}}\big((d_n^{(1)})^{-1}\st{n}{1}-\st{\infty}{1}\big)$ is relatively compact with set of limit points $[-\sigma_1,\sigma_1]$.\\
Since $\st{n}{2}$ is diffusive, by Theorem 1, it is easy to see that $(d_n^{(1)})^{-1}\st{n}{2}\xrightarrow{a.s.}0$. Hence by Cram\'{e}r-Wold theorem we have
\begin{equation}\label{pc141}
n^{2p-\frac{3}{2}}\big((d_n^{(1)})^{-1}S_n-S_{p,q_1,q_2}(1,1)\big)T=n^{2p-\frac{3}{2}}\big((d_n^{(1)})^{-1}\st{n}{1}-\st{\infty}{1},(d_n^{(1)})^{-1}\st{n}{2}\big)\xrightarrow{d}(\sigma_1Z,0)
\end{equation}
where $Z$ is a standard normal variable. Notice that $S_n=\frac{1}{\sqrt{2}}\big(\st{n}{1}+\st{n}{2},\st{n}{1}-\st{n}{2}\big)$, hence using continuous mapping theorem on \eqref{pc141} we get \eqref{cor151}.\\
Further notice that the sequence
$$\frac{n^{2p-\frac{1}{2}}}{\sqrt{2\log\log n}}\big((d_n^{(1)})^{-1}S_n-S_{p,q_1,q_2}(1,1)\big)T=\frac{n^{2p-\frac{1}{2}}}{\sqrt{2\log\log n}}\big((d_n^{(1)})^{-1}\st{n}{1}-\st{\infty}{1},(d_n^{(1)})^{-1}\st{n}{2}\big)$$ is relatively compact with set of limit points $M:=[-\sigma_1,\sigma_1]\times\{0\}$. Hence the almost sure set of limit points of the sequence 
$$
\frac{n^{2p-\frac{1}{2}}}{\sqrt{2\log\log n}}\big((d_n^{(1)})^{-1}S_n-S_{p,q_1,q_2}(1,1)\big)
$$
is $MT=\{(x,x)~|~|x|\leq\sigma_1/\sqrt{2}\}$.\qed

\subsubsection*{\textbf{Proof of Corollary 15:}} Note that since $\st{n}{2}$ is diffusive with eigenvalue $\lambda_2=1-2p$, the Corollary follows by Theorem 2.\qed\\

Proof of the results in low memory regime (i.e. when $p\in[\frac{1}{4},1)$) are similar. Notice that in this case $\st{n}{2}$ is super-diffusive and $\st{n}{1}$ is diffusive. One can obtain the results in low memory regime by just swapping $\st{n}{1}$ with $\st{n}{2}$ and mimicking the proofs of high memory regime with some obvious modifications. Thus we omit the details here.

\section{technical results}
\begin{proposition}
Suppose $u_n$ and $v_n$ are sequence of positive real numbers such that $u_n=o(v_n)$ and $\sum\limits_{n\geq1}v_n=\infty$ then 
$$\sum\limits_{j=1}^nu_j=o\bigg(\sum\limits_{j=1}^nv_j\bigg)$$
\end{proposition}

\begin{proof}
Let $a_n=\sum\limits_{j=1}^nu_j$ and $b_n=\sum\limits_{j=1}^nv_j$, then $b_n$ strictly increases to $\infty$ and 
$$\lim\limits_{n\rightarrow\infty}\frac{a_{n+1}-a_n}{b_{n+1}-b_n}=0$$
hence by Stolz-Ces\`{a}ro theorem 
$\lim\limits_{n\rightarrow\infty}\frac{a_n}{b_n}=0$.
\end{proof}

\begin{proposition}
Suppose $\lambda\in\mathbb{C}$ such that $|\lambda|\leq1$. Then there exist constants $c_{\lambda},c_{\lambda}',c_{\lambda}''>0$ depending only on $\lambda$ such that for all $1\leq j<n,$\\
\begin{enumerate}
\item[(2.1)]~$\big|\prod\limits_{l=j+1}^n\big(\frac{l-1}{l}+\frac{\lambda}{l}\big)\big|\leq c_{\lambda}\big(\frac{j}{n}\big)^{1-\Re(\lambda)}$\\
\item[(2.2)]~$\big|\prod\limits_{l=j+1}^n(\frac{l-1}{l}+\frac{\lambda}{l})-(\frac{j}{n})^{1-\lambda}\big|\leq c_{\lambda}'\frac{1}{j^2}(\frac{j}{n})^{1-\Re(\lambda)}$\\
\item[(2.3)]~$\sup\limits_{x\in[j,j+1]}\big|(\frac{j}{n})^{-\lambda}-(\frac{x}{n})^{-\lambda}\big|\leq c_{\lambda}''\frac{1}{j}(\frac{j}{n})^{-\Re(\lambda)}$
\end{enumerate}
\end{proposition}

\begin{proof}
For $n\geq2$ let $a_n=\prod\limits_{l=2}^n\big(\frac{l-1}{l}+\frac{\lambda}{l}\big)=\prod\limits_{l=2}^n\big(1-\frac{1-\lambda}{l}\big)$. Note that if $\lambda=-1$ then $a_n=0$, in other words, (2.1) and (2.2) holds trivially for $j=1$, so assume that $\lambda\neq-1$. 

By Euler's definition of Gamma function it is easy to see that, as $n\rightarrow\infty$
$$a_n\sim\frac{n^{\lambda-1}}{\Gamma(\lambda+1)}$$ 
taking modulus we have $|a_n|\sim\frac{n^{\Re(\lambda)-1}}{|\Gamma(\lambda+1)|}$, hence we can get constants $c_0,c_1>0$ depending only on $\lambda$, such that for all $n\geq2$
$$c_0n^{\Re(\lambda)-1}\leq a_n\leq c_1n^{\Re(\lambda)-1}$$
now (2.1) follows by choosing $c_{\lambda}=\frac{c_1}{c_0}$. For (2.2) note that
\begin{equation}\label{p2}
\begin{split}
\frac{a_{n+1}}{a_n}-\bigg(\frac{n}{n+1}\bigg)^{1-\lambda} &=\frac{n}{n+1}+\frac{\lambda}{n+1}-\bigg(\frac{n}{n+1}\bigg)^{1-\lambda}\\
&=\bigg(\frac{n}{n+1}\bigg)^{1-\lambda}\bigg[\bigg(1+\frac{1}{n}\bigg)^{-\lambda}\bigg(1+\frac{\lambda}{n}\bigg)\bigg]
\end{split}
\end{equation}
using the series expansion of the analytic function $\lambda\mapsto\big(1+\frac{1}{n}\big)^{-\lambda}$ we get
$$\bigg(1+\frac{1}{n}\bigg)^{-\lambda}=1-\frac{\lambda}{n}+O\bigg(\frac{1}{n^2}\bigg)$$
using this expression in the last line of \eqref{p2} we get
\begin{equation}\label{p22}
\frac{a_{n+1}}{a_n}=\bigg(\frac{n}{n+1}\bigg)^{1-\lambda}\bigg[1+O\bigg(\frac{1}{n^2}\bigg)\bigg]
\end{equation}
now note that
\begin{align*}
\prod\limits_{l=j+1}^n\bigg(\frac{l-1}{l}+\frac{\lambda}{l}\bigg) &=\prod\limits_{l=j}^{n-1}\frac{a_{l+1}}{a_l}\\
&=\bigg(\frac{j}{n}\bigg)^{1-\lambda}\prod\limits_{l=j}^{n-1}\bigg[1+O\bigg(\frac{1}{l^2}\bigg)\bigg]
\end{align*}
Observe that the infinite product $\prod\limits_{l=1}^{\infty}\big|1+O\big(\frac{1}{l^2}\big)\big|$ converges, so we must have $$\prod\limits_{l=j}^{n-1}\bigg[1+O\bigg(\frac{1}{l^2}\bigg)\bigg]=1+O\bigg(\frac{1}{j^2}\bigg)$$ this completes the proof of (2.2).\\

Lastly for (2.3) let $x\in[j,j+1]$. Note that as $j\rightarrow\infty$ 
$$\bigg(\frac{x}{j}\bigg)^{-\lambda}=1-\lambda\left(\frac{x-j}{j}\right)+O\bigg(\bigg(\frac{x-j}{j}\bigg)^2\bigg)$$
since $0\leq x-j\leq1$ one can always find constant $c_{\lambda}''>0$ , free from $x$ and $j$ such that
$$\bigg|1-\bigg(\frac{x}{j}\bigg)^{-\lambda}\bigg|\leq\frac{c_{\lambda}''}{j},~\text{for all}~j\geq1$$
Now (2.3) follows by noticing that $\big|(\frac{j}{n})^{-\lambda}-(\frac{x}{n})^{-\lambda}\big|=(\frac{j}{n})^{-\Re(\lambda)}\big|1-(\frac{x}{j})^{-\lambda}\big|$.
\end{proof}

\begin{proposition}
Suppose $B$ is diagonalisable and matrices $\C{j}{n}$ are defined as in \eqref{C}, then there exist $c,c'>0$ depending only on $B$ such that for all $1\leq j\leq n$\\
\begin{enumerate}
\item[(3.1)] $\Vert\C{j}{n}\rVert\leq c\big(\frac{j}{n}\big)^{1-\eta}$\\

\item[(3.2)] $\lVert\C{j}{n}-(\frac{j}{n})^{I-B}\rVert\leq c'\frac{1}{j^2}(\frac{j}{n})^{1-\eta}$\\
\end{enumerate}
Further there exists $c''>0$ depending only on $B$ such that for $1\leq j<n$\\
\begin{enumerate}
\item[(3.3)] $\sup\limits_{x\in[j,j+1]}\lVert(\frac{j}{n})^{-B}-(\frac{x}{n})^{-B}\rVert\leq c''\frac{1}{j}(\frac{j}{n})^{-\eta}$
\end{enumerate}
\end{proposition}
\begin{proof}
Note that 
\begin{equation*}
\C{j}{n}=TA_{j,n}T^{-1}
\end{equation*} 
where \begin{equation*}
A_{j,n}=\begin{cases}
\text{diag}\bigg(\prod\limits_{l=j+1}^n\bigg(\frac{l-1}{l}+\frac{\lambda_1}{l}\bigg),...,\prod\limits_{l=j+1}^n\bigg(\frac{l-1}{l}+\frac{\lambda_k}{l}\bigg)\bigg) &,\text{if}~1\leq j<n\\[0.25cm] 
I &,\text{if}~j=n
\end{cases}
\end{equation*}
Hence for $1\leq j<n$
\begin{equation*}
\lVert A_{j,n}\rVert=\max\limits_{1\leq r\leq k}\bigg|\prod\limits_{l=j+1}^n\bigg(\frac{l-1}{l}+\frac{\lambda_r}{l}\bigg)\bigg|\leq\max\limits_{1\leq r\leq k}c_{\lambda_r}\bigg(\frac{j}{n}\bigg)^{1-\eta}
\end{equation*}
where $c_{\lambda_r}$'s are defined as in (2.1) of Proposition 2.\\

Now (3.1) follows by setting $c=\max\limits_{1\leq r\leq k}c_{\lambda_r}\lVert T\rVert\lVert T^{-1}\rVert\vee1$ and noticing that 
$$\lVert\C{j}{n}\rVert\leq\lVert T\rVert\lVert T^{-1}\rVert\lVert A_{j,n}\rVert$$
(3.2) and (3.3) follow similarly using (2.2) and (2.3) of Proposition 2 respectively. 
\end{proof}

\bibliographystyle{unsrt}
\bibliography{ERWG}
\end{document}